\documentclass[11pt]{article}%
\usepackage[fleqn]{amsmath}
\usepackage{amsmath,amsfonts,amssymb,graphicx,makeidx,amsthm}
\usepackage{amscd,amsfonts,amssymb,multicol}
\usepackage{amssymb, amsmath}
\usepackage{amsfonts}
\usepackage{amssymb}
\usepackage{graphicx}%
\providecommand{\U}[1]{\protect\rule{.1in}{.1in}}
\newtheorem{theorem}{Theorem}[section]

\newtheorem{corollary}[theorem]{Corollary}

\newtheorem{problem}[theorem]{Problem}
\newtheorem{proposition}[theorem]{Proposition}
\newtheorem{remark}[theorem]{Remark}

\numberwithin{equation}{section}

\begin{document}

\title{A practical method for recovering Sturm-Liouville problems from the Weyl function}
\author{Vladislav V. Kravchenko$^{1}$, Sergii M. Torba$^{1}$\\{\small $^{1}$ Departamento de Matem\'{a}ticas, Cinvestav, Unidad
Quer\'{e}taro, }\\{\small Libramiento Norponiente \#2000, Fracc. Real de Juriquilla,
Quer\'{e}taro, Qro., 76230 MEXICO.}\\{\small e-mail: vkravchenko@math.cinvestav.edu.mx,
storba@math.cinvestav.edu.mx}}
\maketitle

\begin{abstract}
In the paper we propose a direct method for recovering the Sturm-Liouville
potential from the Weyl-Titchmarsh $m$-function given on a countable set of
points. We show that using the Fourier-Legendre series expansion of the
transmutation operator integral kernel the problem reduces to an infinite
linear system of equations, which is uniquely solvable if so is the original
problem. The solution of this linear system allows one to reconstruct the
characteristic determinant and hence to obtain the eigenvalues as its zeros
and to compute the corresponding norming constants. As a result, the original
inverse problem is transformed to an inverse problem with a given spectral
density function, for which the direct method of solution from \cite{KT2020}
is applied.

The proposed method leads to an efficient numerical algorithm for solving a
variety of inverse problems. In particular, the problems in which two spectra
or some parts of three or more spectra are given, the problems in which the
eigenvalues depend on a variable boundary parameter (including spectral
parameter dependent boundary conditions), problems with a partially known
potential and partial inverse problems on quantum graphs.

\end{abstract}

\section{Introduction}

We consider the inverse problem of recovering the potential of a
Sturm-Liouville equation on a finite interval together with the boundary
conditions from values of the Weyl-Titchmarsh $m$-function on a given set of points.

A fundamental result of Marchenko \cite{Marchenko52} states that the
Weyl-Titchmarsh function $m(z)$ determines uniquely the potential $q\in
L_{2}(0,\pi)$ as well as the boundary conditions. So some particular inverse
problem is solved, at least in theory, if one finds a way to convert it to a
corresponding problem for the $m$-function. However, several difficulties
arise. For some inverse spectral problems the given spectral data allow one to
obtain the values of the $m$-function at any point $z\in\mathbb{C}$. This is
the case of the classical inverse problem with a given spectral density
function, formula \eqref{M via spectral data} below provides the $m$-function.
While for many problems only the values of the $m$-function on a countable set
of points $\{z_{k}\}_{k=0}^{\infty}$ can be obtained. See Subsection
\ref{Subsect Example Problems} and references therein for some examples. Thus
one faces two principle questions. Is it possible to recover uniquely the
$m$-function from the given values $m(z_{k})$, $k=0,1,\ldots$? If yes, how can
that be done?

Since the $m$-function is meromorphic, the unique recovery is always possible
if the sequence $\{z_{k}\}_{k=0}^{\infty}$ contains a subsequence $z_{n_{k}}$
converging to a finite limit and such that the values $m(z_{n_{k}})$ are
bounded. This is usually not the case when we start with some inverse problem,
the points $z_{k}$ are real and go to infinity when $k\rightarrow\infty$.
Suppose all the points $z_{k}$ are real. A simple sufficient condition for the
unique recoverability of the $m$-function from the given values $m(z_{k})$ was
given in \cite{GRS1997}, and a complete answer was given not so long ago in
\cite{Horvath2005}. Simply speaking, the points $z_{k}$ should be distributed
sufficiently densely, like the points obtained as a union of two spectra.

The corresponding results are obtained by using infinite products and analytic
continuation. And as it is stated in the following quote \cite[p.106]{Chadan
et al 1997},

\begin{quote}
It is one of the guiding principles of computational analysis that such proofs
often lead to elegant mathematics but are useless as far as constructive
methods are concerned.
\end{quote}

The only constructive method for finding the $m$-function we are aware of is
that from \cite{Bondarenko2020}, \cite{YBX2020} (based on the ideas of \cite{Rundell Sacks}
and \cite[Chapter 3]{Chadan et al 1997}). This algorithm recovers a certain
combination of the integral kernel of the transmutation operator with its
derivative by using non-harmonic Fourier series. However it works only for sequences $z_n$ such that the sequence of functions
$\{e^{\pm i\sqrt{z_n}t}\}_{n\in \mathbb{N}}$ is a Riesz basis in $L_2(-2\pi,2\pi)$,
requires
additionally a beforehand knowledge of the parameter $\omega$ (depending on
the average of the potential $q$ and boundary parameter $H$) and numerical
realization of this method is expected to converge slowly (no numerical
illustration is provided in \cite{Bondarenko2020} or \cite{YBX2020}).

We use Fourier-Legendre series to represent the transmutation integral kernel
and its derivative. This idea was originally proposed in \cite{KNT} and
resulted in a Neumann series of Bessel functions (NSBF) representation for the
solutions and their derivatives of Sturm-Liouville equations. As a result, a
constructive algorithm for recovering the $m$-function from its values on a
countable set of points is proposed. The problem is reduced to an infinite
system of linear equations and the unique solvability of this system is
proved. We would like to emphasize that no additional information is necessary beforehand.

Moreover, the proposed method leads to an efficient numerical algorithm. Since
the NSBF representation possesses such a remarkable property as a uniform
error bound for all real $\lambda>0$, one can obtain any finite set of
approximate eigenvalues and corresponding norming constants with a
non-deteriorating accuracy. That is, an efficient conversion of the original
problem to an inverse problem by a given spectral density function is
proposed. The method of solution introduced in \cite{Kr2019JIIP},
\cite{KrBook2020} and refined in \cite{KT2020} is adapted to recover the
potential and the boundary conditions from the spectral density function.

For classical inverse spectral problems there are many results describing to
what extent the potential can be recovered from a finite set of spectral data,
see, e.g., \cite{Hochstadt1973}, \cite{Neher1994}, \cite{SSh2014},
\cite{Savchuk2016}. For the Weyl function, on the contrary, we are not aware
of any single result. It can be seen that the values of $z_{k}$, $0\leq k\leq
K$ can be chosen neither from a too small interval nor too sparse. Indeed,
taking all the values $z_{k}$ from one spectrum is insufficient to recover the
potential. On the other hand, as was mentioned in \cite[Subsections 3.10.5 and
3.14]{Chadan et al 1997}, when all the values $z_{k}$ are equal to the first
eigenvalue $\lambda_{0}$ (for problems with different boundary conditions),
the corresponding numerical problem is extremely ill-conditioned. So the
proposed algorithm struggles as well when the points $z_{k}$ either belong to
a small interval or are too sparse. Nevertheless, it delivers excellent
results for smooth potentials and sufficiently dense points $z_{k}$, does not
require any additional information on the potential or boundary conditions and
can be applied to a variety of inverse problems for which not so many (if any)
specialized algorithms are known. So we believe the presented method will be a
starting point for the future research of the question \textquotedblleft to
which extent a potential can be recovered from a given finite set of values of
the Weyl-Titchmarsh $m$-function\textquotedblright, as well as of a new class
of numerical methods based on the Weyl-Titchmarsh $m$-function.

The paper is organized as follows. In Section \ref{Sect2} we provide necessary
information on the Weyl function, transmutation operators and Neumann series
of Bessel functions representations. In Section \ref{Sect3} we show how the
inverse problem of recovering the Sturm-Liouville problem from the Weyl
function is reduced to an infinite system of linear algebraic equations, prove
its unique solvability and show the reduction of the original problem to the
classical problem of the recovery of a Sturm-Liouville problem from a given
spectral density function; an algorithm is provided. Finally in Section
\ref{Sect Numerical examples} we present several numerical examples
illustrating the efficiency and some limitations of the proposed method.

\section{Preliminaries}

\label{Sect2}

\subsection{The Weyl function and spectral data}

Let $q\in L_{2}(0,\pi)$ be real valued. Consider the Sturm-Liouville equation
\begin{equation}
-y^{\prime\prime}+q(x)y=\rho^{2}y,\quad x\in(0,\pi), \label{SL equation}%
\end{equation}
where $\rho\in\mathbb{C}$, with the boundary conditions
\begin{equation}
y^{\prime}(0)-hy(0)=0,\quad y^{\prime}(\pi)+Hy(\pi)=0, \label{bc1}%
\end{equation}
where $h$ and $H$ are arbitrary real constants, and the boundary value
problem
\begin{equation}
y^{\prime}(0)-hy(0)=1,\quad y^{\prime}(\pi)+Hy(\pi)=0. \label{bvp1}%
\end{equation}

If $\lambda= \rho^{2}$ is not an eigenvalue of the spectral problem
\eqref{SL equation}, \eqref{bc1}, then the boundary value problem
\eqref{SL equation}, \eqref{bvp1} possesses a unique solution, which we denote
by $\Phi(\rho, x)$. Let
\[
M(\lambda) := \Phi(\rho, 0).
\]
The functions $\Phi(\rho,x)$ and $M(\lambda)$ are called the Weyl solution and
the Weyl function, respectively. See \cite[Section 1.2.4]{Yurko2007} for
additional details.

By $\varphi(\rho,x)$ and $S(\rho,x)$ we denote the solutions of
\eqref{SL equation} satisfying the initial conditions
\begin{equation}
\varphi(\rho,0)=1\quad\text{and}\quad\varphi^{\prime}(\rho,0)=h
\label{init cond}%
\end{equation}
and
\begin{equation}
S(\rho,0)=0\quad\text{and}\quad S^{\prime}(\rho,0)=1, \label{init cond2}%
\end{equation}
respectively. Denote additionally for $\lambda=\rho^{2}$
\begin{equation}
\Delta(\lambda):=\varphi^{\prime}(\rho,\pi)+H\varphi(\rho,\pi)\quad
\text{and}\quad\Delta^{0}(\lambda):=S^{\prime}(\rho,\pi)+HS(\rho,\pi).
\label{CharFuns}%
\end{equation}

Obviously, for all $\rho\in\mathbb{C}$ the function $\varphi(\rho,x)$ fulfills
the first boundary condition, $\varphi^{\prime}(\rho,0)-h\varphi(\rho,0)=0$,
and thus, the spectrum of problem \eqref{SL equation}, \eqref{bc1} is the
sequence of numbers $\left\{  \lambda_{n}=\rho_{n}^{2}\right\}  _{n=0}%
^{\infty}$ such that
\[
\Delta(\lambda_{n})=0.
\]

It is easy to see that
\begin{equation}
\label{Phi via M}\Phi(\rho,x) = S(\rho, x) + M(\lambda)\varphi(\rho,x)
\end{equation}
and
\begin{equation}
\label{M fraction}M(\lambda) = -\frac{\Delta^{0}(\lambda)}{\Delta(\lambda)}.
\end{equation}
Moreover, the Weyl function is meromorphic with simple poles at $\lambda=
\lambda_{n}$, $n\ge0$.

\begin{remark}
There is another common definition of the Weyl-Titchmarsh $m$-function
\cite{Simon1999}, \cite{GS2000}, \cite{Horvath2005}. Let $v(\rho,x)$ be the
solution of \eqref{SL equation} satisfying
\begin{equation}
\label{cond for m}v(\rho, \pi) = 1\quad\text{and}\quad v^{\prime}(\rho,\pi) =
-H.
\end{equation}
Then the Weyl-Titchmarsh $m$-function is defined as
\begin{equation}
\label{m function}m(\rho^{2}) = \frac{v^{\prime}(\rho,0)}{v(\rho,0)}.
\end{equation}

One can verify that
\begin{equation}
m(\rho^{2})=h-\frac{\Delta(\rho^{2})}{\Delta^{0}(\rho^{2})}=h+\frac{1}%
{M(\rho^{2})}, \label{m and M}%
\end{equation}
that is, the functions $m$ and $M$ can be easily transformed one into another.
\end{remark}

Denote additionally
\begin{equation}
\label{alpha n}\alpha_{n}:=\int_{0}^{\pi}\varphi^{2}(\rho_{n},x)\,dx.
\end{equation}
The set $\left\{  \alpha_{n}\right\}  _{n=0}^{\infty}$ is referred to as the
sequence of norming constants of problem \eqref{SL equation}, \eqref{bc1}.

\begin{theorem}
[{see \cite[Theorem 1.2.6]{Yurko2007}}]The following representation holds
\begin{equation}
\label{M via spectral data}M(\lambda) = \sum_{n=0}^{\infty}\frac{1}{\alpha_{n}
(\lambda-\lambda_{n})}.
\end{equation}

\end{theorem}

Let us formulate the inverse Sturm-Liouville problem.

\begin{problem}
[Recovery of a Sturm-Liouville problem from its Weyl function]%
\label{Problem Weyl} Given the Weyl function $M(\lambda)$, find a real valued
$q\in L_{2}(0,\pi)$, and the constants $h$, $H\in\mathbb{R}$, such that
$M(\lambda)$ be the Weyl function of problem \eqref{SL equation}, \eqref{bc1}.
\end{problem}

We refer the reader to \cite[Theorem 1.2.7]{Yurko2007} for the result on the
unique solvability of Problem \ref{Problem Weyl}.

Having the Weyl function known completely, one can recover the sequence
$\{\lambda_{n}\}_{n=0}^{\infty}$ as the sequence of its poles, and the
sequence $\{\alpha_{n}\}_{n=0}^{\infty}$ can be immediately obtained from the
corresponding residuals, the formula (1.2.14) from \cite{Yurko2007} states
that
\begin{equation}
\alpha_{n}=\frac{1}{\operatorname{Res}_{\lambda=\lambda_{n}}M(\lambda)}%
=-\frac{1}{\Delta^{0}(\lambda_{n})}\cdot\left.  \frac{d}{d\lambda}%
\Delta(\lambda)\right\vert _{\lambda=\lambda_{n}}. \label{alphan from M}%
\end{equation}
Hence reducing the inverse problem to recovering the Sturm-Liouville problem
by its spectral density function. Or, by taking the sequences of poles and
zeros of the Weyl function one can reduce Problem \ref{Problem Weyl} to
recovering the Sturm-Liouville problem by two spectra, see \cite[Remark
1.2.3]{Yurko2007}.

However the situation is more complicated if one looks for a practical method
of recovering $q$, $h$ and $H$ from the Weyl function given only on a
countable (even worse, finite) set of points from some (small) interval. The
direct reduction to the inverse problems considered, in particular, in
\cite{KT2020}, is not possible. Let us formulate the corresponding problem and
state some sufficient results for its unique solvability.

\begin{problem}
[Recovery of a Sturm-Liouville problem from its Weyl function given on a
countable set of points]\label{Problem Weyl_countable} Given the values
$M_{n}=M(z_{n})$, $M_{n}\in\mathbb{C}\cup\{\infty\}$, of the Weyl function
$M(\lambda)$ on a set of points $\{z_{n}\}_{n=0}^{\infty}$, find a real valued
$q\in L_{2}(0,\pi)$, and the constants $h$, $H\in\mathbb{R}$, such that the
Weyl function of problem \eqref{SL equation}, \eqref{bc1} takes the values
$M_{n}$ at the points $z_{n}$, $n\geq0$.
\end{problem}

It is easy to see that Problem \ref{Problem Weyl_countable} is not always
uniquely solvable. The following result shows that choosing the set of points
$z_{n}$ \textquotedblleft sufficiently dense\textquotedblright\ can guarantee
the unique solvability.

\begin{theorem}
[{\cite[Theorem 2.1]{GRS1997}}]\label{Thm unique solvability} Let $a_{+} =
\max(a,0)$ for $a\in\mathbb{R}$. Let $\{z_{n}\}_{n=0}^{\infty}$ be a sequence
of distinct positive real numbers satisfying
\begin{equation}
\label{Condition on lambda n}\sum_{n=1}^{\infty}\frac{\left(  z_{n} - \frac14
n^{2}\right)  _{+}}{n^{2}} <\infty.
\end{equation}
Let $m_{1}$ and $m_{2}$ be Weyl-Titchmarsh $m$-functions for two
Sturm-Liouville problems having the potentials $q_{1}$ and $q_{2}$ and the
constants $H_{1}$ and $H_{2}$, correspondingly. Suppose that $m_{1}(z_{n}) =
m_{2}(z_{n})$ for all $n\in\mathbb{N}_{0}$. Then $m_{1} = m_{2}$ (and hence
$q_{1}=q_{2}$ a.e.\ on $[0,\pi]$ and $H_{1}=H_{2}$).
\end{theorem}

See also \cite[Theorem 1.5]{Horvath2005} for the necessary and sufficient result.

\subsection{Some inverse problems which can be reduced to Problem
\ref{Problem Weyl_countable}}

\label{Subsect Example Problems} Below we briefly present some inverse
problems, omitting any additional details like the conditions on a spectra
ensuring the uniqueness of the solution. They can be consulted in the provided references.

\subsubsection{Two spectra}

\label{Subsubsect Two spectrum} Consider the second set of boundary
conditions
\begin{equation}
y(0)=0,\quad y^{\prime}(\pi)+Hy(\pi)=0. \label{bc2}%
\end{equation}
Let us denote the spectrum of problem (\ref{SL equation}), (\ref{bc2}) as
$\left\{  \nu_{n}=\mu_{n}^{2}\right\}  _{n=0}^{\infty}$.

The inverse problem consists in recovering $q$, $h$ and $H$ from given two
spectra, $\{\lambda_{n}\}_{n=0}^{\infty}$ and $\{\nu_{n}\}_{n=0}^{\infty}$.
Note that this problem coincides with the one considered in \cite{KT2020} if
one changes $x$ by $\pi-x$ and $h\leftrightarrow H$. Now the equalities
\[
M(\lambda_{n})=\infty\quad\text{and}\quad M(\nu_{n})=0,
\]
can be easily checked and give the reduction of the two spectra inverse
problem to Problem \ref{Problem Weyl_countable}.

\subsubsection{Eigenvalues depending on a variable boundary parameter}

\label{SubSubsectHn} There is a class of inverse problems where it is asked to
recover the potential from given parts of three spectra \cite{GRS1997}, four
spectra \cite{Pivovarchik}, $N$ spectra \cite{Horvath2001}. All these problems
can be regarded as special cases of the following inverse problem. Suppose the
sequences of numbers $\{\tilde{\lambda}_{n}\}_{n=0}^{\infty}$ and
$\{h_{n}\}_{n=0}^{\infty}$ are given and such that $\tilde{\lambda}_{n}$ is an
eigenvalue (does not refer to $n$-th indexed eigenvalue) of the problem
\eqref{SL equation} with the boundary conditions
\begin{equation}
y^{\prime}(0)-h_{n}y(0)=0,\quad y^{\prime}(\pi)+Hy(\pi)=0. \label{bc variable}%
\end{equation}
The problem is to recover the potential $q$. We refer the reader to
\cite[\S 3.10.5 and \S 3.15]{Chadan et al 1997} for some discussion and
numerical results.

The reduction to Problem \ref{Problem Weyl_countable} follows immediately when
one observes that \eqref{bc variable} leads to
\[
m(\tilde\lambda_{n}) = h_{n}.
\]

\subsubsection{Partially known potential}

\label{SubsubsectPartial} The Hochstadt-Lieberman inverse problem from
\cite{HL1978} reads as follows: suppose we are given the spectrum of the
problem \eqref{SL equation}, \eqref{bc1} and the potential $q$ is known on the
segment $[0,\pi/2]$. Recover $q$ on the whole segment $[0,\pi]$. See also
\cite{GS2000-2}, \cite{GRS1997} for the case when the potential is known on a
different portion of the segment. If the potential is known on less than half
of the segment, more than one spectrum is necessary to recover the potential
uniquely. For that reason the general formulation of the problem is as
follows. Suppose $a\in(0,\pi)$ and the potential $q$ is known on $[0,a]$.
Moreover, two sequences of numbers are given as in Subsubsection
\ref{SubSubsectHn}. Recover $q$ on the whole $[0,\pi]$.

Let $m(x,\rho^{2})$ define the $m$-function on $[x,\pi]$ with the same
boundary condition \eqref{cond for m}, i.e.,
\[
m(x,\rho^{2})=\frac{v^{\prime}(\rho,x)}{v(\rho,x)}.
\]
Consider an eigenvalue $\tilde{\lambda}_{n}=\rho_{n}^{2}$ and the
corresponding boundary constant $h_{n}$. Then
\[
m(0,\tilde{\lambda}_{n})=m(\tilde{\lambda}_{n})=h_{n},
\]
hence we may consider the following Cauchy problem for equation
\eqref{SL equation}
\[
v(\rho_{n},0)=1,\quad v^{\prime}(\rho_{n},0)=h_{n}.
\]
Solving it on $[0,a]$ (in practice it can be done efficiently for a large set
of $\tilde{\lambda}_{n}$ using the method developed in \cite{KNT}), we obtain
values $v(\rho_{n},a)$ and $v^{\prime}(\rho_{n},a)$, hence, $m(a,\tilde
{\lambda}_{n})$. And by rescaling we reduce the problem of a partially known
potential to Problem \ref{Problem Weyl_countable}.

\subsubsection{Problems with analytic functions in the boundary condition}

The following spectral parameter dependent boundary condition was considered
in \cite{Bondarenko2020}, \cite{Bondarenko2020-2}, \cite{YBX2020}
\begin{equation}
f_{1}(\lambda)y^{\prime}(0)+f_{2}(\lambda)y(0)=0, \label{bc analytic}%
\end{equation}
where $f_{1}$ and $f_{2}$ are entire functions, not vanishing simultaneously.
The second boundary condition is the same as in \eqref{bc1}. The functions
$f_{1}$ and $f_{2}$ are supposed to be known, and an inverse problem consists
in recovering the potential $q$ and the constant $H$ by the given parameter
$\omega_{2}$ (see \eqref{omega}), spectra $\lambda_{n}$ and corresponding
multiplicities $m_{n}$. See also \cite{Shkalikov1986} for a general spectral
theory of such problems.

Supposing for simplicity that all the eigenvalues are simple, i.e., $m_{n}=1$
for all $n$, the problem immediately reduces to Problem
\ref{Problem Weyl_countable}, since
\[
m(\lambda_{n}) = -\frac{f_{2}(\lambda_{n})}{f_{1}(\lambda_{n})}.
\]

Many inverse problems can be reduced to a problem having boundary condition
\eqref{bc analytic}. In particular, the Hochstadt-Lieberman problem, the
inverse transmission eigenvalue problem, partial inverse problems for quantum
graphs. See \cite{YBX2020} and references therein for more details.

\subsection{Some practical applications}

Inverse Sturm-Liouville problems arise in numerous applied field. For some
examples in vibration models we refer to the book  \cite{GladwellBook}, for
some biomedical engineering applications to the review \cite{Gou Chen 2015},
while here we add two more examples from different areas of physics.

In studying quantum physical properties of quantum dot nanostructures the
problem of recovering symmetric potentials from a finite number of
experimentally obtained eigenvalues is of considerable importance (see, e.g.,
\cite{Khmelnytskaya et al 2010}). Mathematically this means that the potential
$q$ in (\ref{SL equation}) is symmetric: $q(x)=q(\pi-x)$, and a finite part of
one spectrum is given. For simplicity let us suppose that it corresponds to
the Dirichlet conditions. Then, as it is well known (see \cite[p. 103]{Chadan
et al 1997}), this problem reduces to the two spectra problem on the interval
$(0,\pi/2)$. Indeed, for the eigenfunctions we have $y_{n}(x)=y_{n}(\pi-x)$ if
$n$ is odd and $y_{n}(x)=-y_{n}(\pi-x)$ if $n$ is even. Thus, $y_{n}(0)=0$,
$y_{n}(\pi/2)=0$ if $n$ is even and $y_{n}(0)=0$, $y_{n}^{\prime}(\pi/2)=0$ if
$n$ is\ odd. This means that the original spectrum is split into two parts --
the even-numbered eigenvalues forming the Dirichlet-Dirichlet eigenvalues on
$(0,\pi/2)$ and the odd-numbered ones, the Dirichlet-Neumann eigenvalues on
$(0,\pi/2)$. Thus the problem is reduced to the two spectra inverse problem,
and symmetry of $q$ allows one to compute it over the entire interval
$(0,\pi)$.

In order to proceed with the second example, following \cite{Gladwell96}, we
notice that besides the standard form (\ref{SL equation}) of the
Sturm-Liouville equation, which is often the most convenient for analysis, two
other forms
\begin{equation}
\left(  a^{2}(x)u^{\prime}(x)\right)  ^{\prime}+\lambda a^{2}%
(x)u(x)=0\label{rod}%
\end{equation}
and
\begin{equation}
u^{\prime\prime}(x)+\lambda r(x)u(x)=0\label{string}%
\end{equation}
frequently arise in applications. When the functions $a$, $r$, $q$ are
sufficiently smooth, each equation may be transformed into any of the other
two. Equation (\ref{rod}) governs the modes of vibration of a thin straight
rod in longitudinal or torsional vibration. It is often called the
Sturm-Liouville equation in impedance form. Equation (\ref{string}) models the
transverse vibration of a taut string with mass density $r(x)$.

Let us consider the problem of recovering the cross-sectional form $F(x)$ of a
rod with constant density $r$ and Young's modulus $E$ (see \cite[p.72]%
{Vatulyan 2007}). The differential equation and the boundary conditions have
the form%
\begin{gather}
\left(  EF(x)u^{\prime}(x)\right)  ^{\prime}+rF(x)\omega^{2}%
u(x)=0,\label{rodF}\\
u^{\prime}(0)=-\frac{p}{EF(0)},\qquad u^{\prime}(\pi)=0,\label{bcrod}%
\end{gather}
where $p$ is a given constant. The values of $F$ at the end points $F(0)$ and
$F(\pi)$ are supposed to be given, and the function $F(x)$, $x\in(0,\pi)$
needs to be recovered from the additional information on the solution:%
\begin{equation}
u(\omega,0)=\widetilde{f}(\omega),\qquad\omega\in\left[  \omega_{1},\omega
_{2}\right]  .\label{addinfo}%
\end{equation}
Note that for simplicity we substituted the Dirichlet condition at $\pi$ (in
\cite[p.72]{Vatulyan 2007}) by the Neumann condition. Clearly, the problem
with the Dirichlet condition can be considered similarly.

First of all, equation (\ref{rodF}) can be written in the form (\ref{rod})
with $a^{2}=F$ and $\lambda=r\omega^{2}/E$. Next, for the function
$y(x)=u(x)/a(x)$ the problem (\ref{rodF}), (\ref{bcrod}), (\ref{addinfo})
takes the form%
\begin{gather}
-y^{\prime\prime}+q(x)y=\rho^{2}y,\qquad x\in(0,\pi),\nonumber\\
y^{\prime}(\rho,0)-hy(\rho,0)=c,\qquad y^{\prime}(\rho,\pi)+Hy(\rho
,\pi)=0,\label{bcond}\\
y(\rho,0)=f(\rho),\qquad\rho\in\left[  \rho_{1},\rho_{2}\right] ,\nonumber
\end{gather}
where $h=-a^{\prime}(0)/a(0)$, $c=-p/(EF^{3/2}(0))$, $H=a^{\prime}(\pi
)/a(\pi)$ and $f(\rho)=\widetilde{f}(\omega)/a(0)$. The boundary conditions
(\ref{bcond}) indicate that $y(\rho,x)=c\Phi(\rho,x)$. Thus, the original
problem reduces to the problem of recovering the potential $q$ and the
constants $h$, $H$ from the Weyl function
\[
M(\lambda)=\frac{f(\rho)}{c}%
\]
given on the segment $\left[  \rho_{1},\rho_{2}\right]  $.

\subsection{Integral representation of solutions via the transmutation
operator}

The solutions $\varphi(\rho,x)$ and $S(\rho,x)$ admit the integral
representations (see, e.g., \cite{Chadan}, \cite{LevitanInverse},
\cite{Marchenko52}, \cite[Chapter 1]{Marchenko}, \cite{SitnikShishkina
Elsevier}, \cite{CKT})
\begin{align}
\varphi(\rho, x)  &  = \cos\rho x + \int_{-x}^{x} K(x,t) \cos\rho t\,
dt,\label{phi via T}\\
S(\rho, x)  &  = \frac{\sin\rho x}{\rho} + \int_{-x}^{x} K(x,t) \frac{\sin\rho
t}{\rho} \,dt, \label{S via T}%
\end{align}
where the integral kernel $K$ is a continuous function of both arguments in
the domain $0\leq|t|\leq x\leq\pi$ and satisfies the equalities
\begin{equation}
K(x,x)=\frac h2+\frac{1}{2}\int_{0}^{x}q(t)\,dt\quad\text{and}\quad
K(x,-x)=\frac h2. \label{Kxx}%
\end{equation}
It is of crucial importance that $K(x,t)$ is independent of $\rho$.

The following result from \cite{KNT} will be used.

\begin{theorem}
[\cite{KNT}]\label{Theorem K series representation} The integral transmutation
kernel $K(x,t)$ and its derivative $K_{1}(x,t):=\frac{\partial}{\partial
x}K(x,t)$ admit the following Fourier-Legendre series representations
\begin{align}
K(x,t)  &  =\sum_{n=0}^{\infty} \frac{\beta_{n}(x)}{x}P_{n}\left(  \frac{t}%
{x}\right)  ,\label{K Fourier series}\\
K_{1}(x,t)  &  =\sum_{n=0}^{\infty}\frac{\gamma_{n}(x)}{x}P_{n}\left(
\frac{t}{x}\right)  ,\quad0<|t|\leq x\leq\pi, \label{K1 via gammas}%
\end{align}
where $P_{k}$ stands for the Legendre polynomial of order $k$.

For every $x\in(0,\pi]$ the series converge in the norm of $L_{2}(-x,x)$. The
first coefficients $\beta_{0}(x)$ and $\gamma_{0}(x)$ have the form
\begin{equation}
\beta_{0}(x)=\frac{\varphi(0,x)-1}2,\quad\gamma_{0}(x) =\beta_{0}^{\prime
}(x)-\frac{h}{2}-\frac{1}{4}\int_{0}^{x}q(s)\,ds, \label{beta0}%
\end{equation}
and the rest of the coefficients can be calculated following a simple
recurrent integration procedure.
\end{theorem}

\begin{corollary}
\label{Rem q from beta0}Equality \eqref{beta0} shows that the potential $q(x)$
can be recovered from $\beta_{0}(x)$. Indeed,
\begin{equation}
q(x)=\frac{\varphi^{\prime\prime}(0,x)}{\varphi(0,x)}=\frac{2\beta_{0}
^{\prime\prime}(x)}{2\beta_{0}(x)+1}. \label{q from beta0}%
\end{equation}
Moreover, the constant $h$ is also recovered directly from $\beta_{0}(x)$,
since%
\[
h=2\beta_{0}^{\prime}(0).
\]

\end{corollary}

\subsection{Neumann series of Bessel functions representations for solutions
and their derivatives}

The following series representations for the solutions $\varphi(\rho,x)$ and
$S(\rho,x)$ and for their derivatives were obtained in \cite{KNT}.

\begin{theorem}
[\cite{KNT}]\label{Theorem NSBF} The solutions $\varphi(\rho,x)$ and
$S(\rho,x)$ and their derivatives with respect to $x$ admit the following
series representations
\begin{align}
\varphi(\rho,x)  &  =\cos\rho x+2\sum_{n=0}^{\infty}(-1)^{n}\beta
_{2n}(x)j_{2n}(\rho x),\label{phi}\\
S(\rho,x)  &  =\frac{\sin\rho x}{\rho}+\frac{2}{\rho}\sum_{n=0}^{\infty
}(-1)^{n}\beta_{2n+1}(x)j_{2n+1}(\rho x),\label{S}\\
\varphi^{\prime}(\rho,x)  &  =-\rho\sin\rho x+\left(  h+\frac{1}{2}\int
_{0}^{x}q(s)\,ds\right)  \cos\rho x+2\sum_{n=0}^{\infty}(-1)^{n}\gamma
_{2n}(x)j_{2n}(\rho x),\label{phiprime}\\
S^{\prime}(\rho,x)  &  =\cos\rho x+\frac{1}{2\rho}\left(  \int_{0}%
^{x}q(s)\,ds\right)  \sin\rho x+\frac{2}{\rho}\sum_{n=0}^{\infty}%
(-1)^{n}\gamma_{2n+1}(x)j_{2n+1}(\rho x), \label{Sprime}%
\end{align}
where $j_{k}(z)$ stands for the spherical Bessel function of order $k$ (see,
e.g., \cite{AbramowitzStegunSpF}). The coefficients $\beta_{n}(x)$ and
$\gamma_{n}(x)$ are those from Theorem \ref{Theorem K series representation}.
For every $\rho\in\mathbb{C}$ all the series converge pointwise. For every
$x\in\left[  0,\pi\right]  $ the series converge uniformly on any compact set
of the complex plane of the variable $\rho$, and the remainders of their
partial sums admit estimates independent of $\operatorname{Re}\rho$.

Moreover, the representations can be differentiated with respect to $\rho$
resulting in
\begin{align}
\varphi^{\prime}_{\rho}(\rho, x)  &  = - x\sin\rho x + 2\sum_{n=0}^{\infty
}(-1)^{n} \beta_{2n}(x) \left(  \frac{2n}\rho j_{2n}(\rho x) - x j_{2n+1}(\rho
x)\right)  ,\label{phiprime rho}\\%
\begin{split}
\varphi^{\prime\prime}_{x,\rho}(\rho, x)  &  = -\left(  1+ hx+\frac{x}{2}%
\int_{0}^{x}q(s)\,ds\right)  \sin\rho x-\rho x \cos\rho x\\
&  \quad+2\sum_{n=0}^{\infty}(-1)^{n} \gamma_{2n}(x) \left(  \frac{2n}\rho
j_{2n}(\rho x) - x j_{2n+1}(\rho x)\right)  .
\end{split}
\label{phiprimeprime}%
\end{align}

\end{theorem}

We refer the reader to \cite{KNT} for the proof and additional details related
to this result. The coefficients $\beta_{n}$ and $\gamma_{n}$ decay when
$n\rightarrow\infty$, and the decay rate is faster for smoother potentials.
Namely, the following result is valid.

\begin{proposition}
[\cite{KrT2018}]\label{Prop coeff decay} Let $q\in W_{\infty}^{p}[0,\pi]$ for
some $p\in\mathbb{N}_{0}$, i.e., the potential $q$ is $p$ times differentiable
with the last derivative being bounded on $[0,\pi]$. Then there exist
constants $c$ and $d$, independent of $N$, such that
\[
|\beta_{N}(x)|\leq\frac{cx^{p+2}}{(N-1)^{p+1/2}}\qquad\text{and}\qquad
|\gamma_{N}(x)|\leq\frac{dx^{p+1}}{(N-1)^{p-1/2}},\qquad N\geq p+1.
\]

\end{proposition}

Let us denote
\begin{equation}
\label{h_n}h_{n} := \gamma_{n}(\pi) + H\beta_{n}(\pi),\qquad n\ge0
\end{equation}
and
\begin{equation}
\omega:=h+H+\frac{1}{2}\int_{0}^{\pi}q(t)\,dt\quad\text{and}\quad\omega
_{2}:=\omega-h = H+\frac12\int_{0}^{\pi}q(t)\,dt. \label{omega}%
\end{equation}

\begin{corollary}
The following representations hold for the characteristic functions $\Delta$
and $\Delta^{0}$
\begin{align}
\Delta(\rho^{2})  &  = -\rho\sin\rho\pi+ \omega\cos\rho\pi+ 2\sum
_{n=0}^{\infty}(-1)^{n}h_{2n}j_{2n}(\rho\pi),\label{Delta NSBF}\\
\Delta^{0}(\rho^{2})  &  = \cos\rho\pi+ \omega_{2} \frac{\sin\rho\pi}{\rho} +
\frac2\rho\sum_{n=0}^{\infty}(-1)^{n}h_{2n+1}j_{2n+1}(\rho\pi),
\label{Delta0 NSBF}%
\end{align}
and for the derivative with respect to $\rho$
\begin{equation}
\label{Delta prime NSBF}\frac{d}{d\rho}\Delta(\rho^{2}) = -(1+\pi\omega
)\sin\rho\pi-\pi\rho\cos\rho\pi+ 2\sum_{n=0}^{\infty}(-1)^{n} h_{2n} \left(
\frac{2n}\rho j_{2n}(\rho\pi) - \pi j_{2n+1}(\rho\pi)\right)  .
\end{equation}

\end{corollary}

\subsection{The Gelfand-Levitan equation}

Let
\[
G(x,t) := K(x,t) + K(x,-t) = \sum_{n=0}^{\infty} \frac{2\beta_{2n}(x)}%
{x}P_{2n}\left(  \frac{t}{x}\right)
\]
and let
\begin{equation}
F(x,t)=\sum_{n=0}^{\infty}\left(  \frac{\cos\rho_{n}x\cos\rho_{n}t}{\alpha
_{n}}-\frac{\cos nx\cos nt}{\alpha_{n}^{0}}\right)  ,\quad0\leq t,\,x<\pi
\label{F}%
\end{equation}
where
\[
\alpha_{n}^{0}=%
\begin{cases}
\pi/2, & n>0,\\
\pi, & n=0.
\end{cases}
\]

Then the function $G$ is the unique solution of the Gelfand-Levitan equation,
see, e.g., \cite[Theorem 1.3.1]{Yurko2007}
\begin{equation}
G(x,t)+F(x,t)+\int_{0}^{x}F(t,s)G(x,s)\,ds=0,\quad0<t<x<\pi.
\label{Gelfand-Levitan}%
\end{equation}

The series \eqref{F} converges slowly and possesses a jump discontinuity (for
$\omega\ne0$) at $x=t=\pi$. To overcome these difficulties, in
\cite{KKT2019AMC} another representation for the function $F$ was derived.
Namely,
\begin{equation}%
\begin{split}
F(x,t)  &  =\sum_{n=1}^{\infty}\left(  \frac{\cos\rho_{n}x\,\cos\rho_{n}%
t}{\alpha_{n}}-\frac{\cos nx\,\cos nt}{\alpha_{n}^{0}}+\frac{2\omega}{\pi
^{2}n}\Bigl(x\sin nx\,\cos nt+t\sin nt\,\cos nx\Bigr)\right) \\
&  \quad+\frac{\cos\rho_{0} x\,\cos\rho_{0} t}{\alpha_{0}}-\frac{1}{\pi}%
-\frac{\omega}{\pi^{2}}\bigl(\pi\max\{x,t\}-x^{2}-t^{2}\bigr).
\end{split}
\label{F alt}%
\end{equation}

\section{Method of solution of Problem \ref{Problem Weyl_countable}}

\label{Sect3}

\subsection{Recovery of the parameters $\omega$ and $\omega_{2}$ and of the
characteristic functions $\Delta$ and $\Delta^{0}$}

\label{Subsect Recovery hn} Let us rewrite \eqref{M fraction} as
\[
M(\rho^{2})\Delta(\rho^{2}) + \Delta^{0}(\rho^{2}) = 0
\]
and substitute \eqref{Delta NSBF} and \eqref{Delta0 NSBF} into the last
expression. We obtain
\[%
\begin{split}
M(\rho^{2})  &  \left(  -\rho\sin\rho\pi+ \omega\cos\rho\pi+ 2\sum
_{n=0}^{\infty}(-1)^{n}h_{2n}j_{2n}(\rho\pi)\right) \\
&  +\cos\rho\pi+ \omega_{2} \frac{\sin\rho\pi}{\rho} + \frac2\rho\sum
_{n=0}^{\infty}(-1)^{n}h_{2n+1}j_{2n+1}(\rho\pi) = 0,
\end{split}
\]
or
\begin{equation}
\label{Equation for hn}%
\begin{split}
\omega\cdot M(\rho^{2})\cos\rho\pi &  + \omega_{2}\cdot\frac{\sin\rho\pi}%
{\rho} + 2\sum_{n=0}^{\infty} h_{2n} \cdot(-1)^{n}M(\rho^{2})j_{2n}(\rho\pi)
+2\sum_{n=0}^{\infty} h_{2n+1} \cdot\frac{(-1)^{n}j_{2n+1}(\rho\pi)}\rho\\
&  = M(\rho^{2}) \rho\sin\rho\pi- \cos\rho\pi.
\end{split}
\end{equation}
For the case when $M(\rho^{2}) = \infty$, the last equality reduces to
\begin{equation}
\label{Equation for hn_infty}\omega\cos\rho\pi+ 2\sum_{n=0}^{\infty} h_{2n}
\cdot(-1)^{n}j_{2n}(\rho\pi) = \rho\sin\rho\pi.
\end{equation}

Suppose that the Weyl function is known on a countable set of points
$\{z_{k}\}_{k=0}^{\infty}=\{\tilde{\rho}_{k}^{2}\}_{k=0}^{\infty}$, that is,
let the values $M_{k}:=M(z_{k})$, $M_{k}\in\mathbb{R}\cup\{\infty\}$ be given.
Then considering the equality \eqref{Equation for hn} for all $\rho
=\tilde{\rho}_{k}$ we obtain an infinite system of linear algebraic equations
for\ the unknown $\omega$, $\omega_{2}$ and $\{h_{n}\}_{n=0}^{\infty}$.

One can easily see from \eqref{K Fourier series} and \eqref{K1 via gammas}
that the functions $\frac{\sqrt{2\pi} h_{n}}{\sqrt{2n+1}}$ are the Fourier
coefficients of the square integrable function $K_{1}(\pi,t) + HK(\pi,t)$ in
the space $L_{2}(-\pi,\pi)$ with respect to the orthonormal basis
\begin{equation}
\label{basis pm}\left\{  p_{n}(t)\right\}  _{n=0}^{\infty}:=\left\{
\frac{\sqrt{2n+1} P_{n}(t/\pi)}{\sqrt{ 2\pi}}\right\}  _{n=0}^{\infty}.
\end{equation}
Let us introduce new unknowns
\begin{equation}
\label{xi_n}\xi_{n} = \frac{\sqrt{2\pi} h_{n}}{\sqrt{2n+1}},\qquad n\ge0.
\end{equation}
Then $\{\xi_{n}\}_{n=0}^{\infty}\in\ell_{2}$ (as a sequence of the Fourier
coefficients). Let us rewrite the infinite linear system of equations as
\begin{equation}
\label{System for xin}%
\begin{split}
&  \sum_{n=0}^{\infty} \xi_{2n} \cdot\frac{(-1)^{n}\sqrt{4n+2} M_{k}
j_{2n}(\tilde\rho_{k} \pi)}{\sqrt\pi} +\sum_{n=0}^{\infty} \xi_{2n+1}
\cdot\frac{(-1)^{n}\sqrt{4n+2}j_{2n+1}(\tilde\rho_{k} \pi)}{\sqrt\pi\tilde
\rho_{k}}\\
&  +\omega\cdot M_{k} \cos\tilde\rho_{k} \pi+ \omega_{2}\cdot\frac{\sin
\tilde\rho_{k}\pi}{\tilde\rho_{k}} = M_{k} \tilde\rho_{n}\sin\tilde\rho_{k}%
\pi- \cos\tilde\rho_{k}\pi, \qquad k\in\mathbb{N}_{0}%
\end{split}
\end{equation}
(equation
\begin{equation}
\label{System for xin_infinity}\sum_{n=0}^{\infty} \xi_{2n} \cdot
\frac{(-1)^{n}\sqrt{4n+2} j_{2n}(\tilde\rho_{k} \pi)}{\sqrt\pi} +\omega
\cdot\cos\tilde\rho_{k} \pi= \tilde\rho_{k}\sin\tilde\rho_{k}\pi
\end{equation}
should be used if for some $k$ we have $M_{k}=\infty$).

\begin{theorem}
Suppose that the original Problem \ref{Problem Weyl_countable} is solvable and
the numbers $\{z_{k}\}_{k=0}^{\infty}$ satisfy the condition
\eqref{Condition on lambda n}. Then the infinite system \eqref{System for xin}
possesses a unique $\ell_{2}$ solution $\left\{  \omega, \omega_{2}, \{\xi
_{n}\}_{n=0}^{\infty}\right\}  $.
\end{theorem}

\begin{proof}
Since the original Problem \ref{Problem Weyl_countable} is solvable, we can recover the Weyl funcion $M$ and by
\cite[Theorem 1.2.7]{Yurko2007} recover the potential $q$ and the constants $h$ and $H$. From \eqref{omega} we obtain $\omega$ and $\omega_2$, and by Theorem \ref{Theorem K series representation} we obtain the square integrable with respect to the second variable integral kernel $K$ and its derivative $K_1$, as well as the series representations \eqref{K Fourier series}, \eqref{K1 via gammas}. The coefficients $\beta_n$ and $\gamma_n$ give us the sequence $\{\xi_n\}_{n=0}^\infty \in \ell_2$. One can verify using \eqref{Delta NSBF} and \eqref{Delta0 NSBF} that the obtained numbers $\omega$, $\omega_2$ and $\xi_n$, $n\ge 0$ satisfy the system \eqref{System for xin}.
Now suppose that there is another $\ell_2$ solution $\left\{\tilde\omega, \tilde\omega_2, \{\tilde\xi_n\}_{n=0}^\infty\right\}$ of the system \eqref{System for xin}. Let $\tilde h_n = \frac{\sqrt{2n+1}}{\sqrt{2\pi}} \tilde \xi_n$, $n\ge 0$. Consider the following functions
\begin{align}
\tilde G(t) & = 2\sum_{n=0}^\infty \frac{\tilde h_{2n}}{\pi}P_{2n}\left(\frac t\pi\right),\label{Ginverse}\\
\tilde S(t) & = 2\sum_{n=0}^\infty \frac{\tilde h_{2n+1}}{\pi}P_{2n+1}\left(\frac t\pi\right),\qquad t\in[0,\pi].\label{Sinverse}
\end{align}
Due to condition $\{\tilde\xi_n\}_{n=0}^\infty \in \ell_2$ one has $\tilde G\in L_2(0,\pi)$ and $\tilde S\in L_2(0,\pi)$, hence the following functions
\begin{align}
\tilde \Delta(\lambda):=\tilde \Delta(\rho^2) &= -\rho\sin\rho\pi + \tilde\omega \cos \rho \pi+ \int_0^\pi \tilde G(t) \cos\rho t\,dt, \qquad \lambda\in\mathbb{C}, \label{Delta inv}\\
\tilde \Delta^0(\lambda):=\tilde \Delta^0(\rho^2) &= \cos \rho \pi + \tilde \omega_2 \frac{\sin \rho\pi}{\rho} + \frac 1\rho\int_0^\pi \tilde S(t) \sin\rho t\,dt, \qquad \lambda\in\mathbb{C}\setminus\{0\} \label{Delta0 inv}
\end{align}
are well defined (to specify $\rho$ we use the square root branch with $\operatorname{Im}\sqrt\lambda \ge 0$). The function $\tilde \Delta^0$ can be continuously extended to $\lambda=0$ as well. The resulting functions $\tilde\Delta$ and $\tilde\Delta^0$ are entire functions of order $1/2$. Moreover, representations \eqref{Delta NSBF} and \eqref{Delta0 NSBF} hold for the functions $\tilde\Delta$ and $\tilde\Delta^0$ if one changes $\omega$, $\omega_2$ and $h_n$ by $\tilde \omega$, $\tilde \omega_2$ and $\tilde h_n$ respectively.
Let us define
\[
\tilde M(\lambda) = -\frac{\tilde\Delta^0(\lambda)}{\tilde \Delta(\lambda)},\qquad \lambda\in\mathbb{C}.
\]
Then one can verify that
\[
\tilde M(z_k) = M_k = M(z_k),\quad k\ge 0.
\]
We would like to emphasize here that we do not know if the function $\tilde M$ is the Weyl function corresponding to some potential $q$, so Theorem \ref{Thm unique solvability} can not be applied directly. Nevertheless, the proof of this theorem in \cite{GRS1997} requires only that $\tilde M$ is a quotient of two entire functions satisfying some basic asymptotic conditions, which can be easily verified for the functions $\tilde \Delta$ and $\tilde\Delta^0$. Hence following the proof of Theorem 2.1 from \cite{GRS1997} we obtain that $\tilde M\equiv M$, a contradiction.
\end{proof}

Now suppose that only a finite set of pairs $\{z_{k},\, M_{k}\}_{k=0}^{K} :=
\{z_{k},\, M(z_{k})\}_{k=0}^{K}$ is given. We assume that the numbers $z_{k}$
are real (the system of equations \eqref{System for xin} can be formulated for
complex numbers $z_{k}$ as well) and ordered: $z_{k} < z_{k+1}$, $0\le k\le
K-1$. Clearly any finite sequence of numbers $z_{k}$, $k\le K$, can be
extended to an infinite one satisfying the condition
\eqref{Condition on lambda n}, which is not of a great use since the truncated
system \eqref{System for xin} for some particular choices of $z_{k}$, $k\le K$
may not be solved uniquely. For example, taking $z_{k} = k^{2}$ for all $k\le
K$ makes it impossible to find the coefficient $\omega_{2}$. As a rule of
thumb we can ask that the terms in \eqref{Condition on lambda n} remain small
and bounded, even better if they decay as $k$ increases. So we can formulate
the following empirical requirement for the numbers $z_{k} =\tilde\rho_{k}%
^{2}$:
\[
\tilde\rho_{k} < \frac k2+1\qquad\text{or equivalently}\qquad z_{k} <
\frac{k^{2}}{4} + k + 1,
\]
at least starting from some $k=K^{\prime}<K$. Such requirement is sufficient
to deal, for example, with the two spectra inverse problem, for which
$\{z_{k}\}_{k=0}^{2K} = \{\lambda_{k}\}_{k=0}^{K} \cup\{\nu_{k}\}_{k=0}^{K}$.

Let us consider a truncated system \eqref{System for xin},
\begin{equation}
\label{System for xin truncated}%
\begin{split}
&  \sum_{n=0}^{[(K-2)/2]} \xi_{2n} \cdot\frac{(-1)^{n}\sqrt{4n+2} M_{k}
j_{2n}(\tilde\rho_{k} \pi)}{\sqrt\pi} +\sum_{n=0}^{[(K-3)/2]} \xi_{2n+1}
\cdot\frac{(-1)^{n}\sqrt{4n+2}j_{2n+1}(\tilde\rho_{k} \pi)}{\sqrt\pi\tilde
\rho_{k}}\\
&  +\omega\cdot M_{k} \cos\tilde\rho_{k} \pi+ \omega_{2}\cdot\frac{\sin
\tilde\rho_{k}\pi}{\tilde\rho_{k}} = M_{k} \tilde\rho_{k}\sin\tilde\rho_{k}%
\pi- \cos\tilde\rho_{k}\pi, \qquad0\le k\le K.
\end{split}
\end{equation}
Solving this system one obtains approximate values of the parameters $\omega$,
$\omega_{2}$ and the coefficients $h_{0},\ldots,h_{K-2}$.

The coefficient matrix of the truncated system
\eqref{System for xin truncated} can be badly conditioned. Since we know that
the coefficients $h_{n}$ decay, see Proposition \ref{Prop coeff decay}, there
is no need to look for the same number of the unknowns as the number of points
$\tilde{\lambda}_{k}$. One may consider less unknowns and treat the system
\eqref{System for xin truncated} as an overdetermined one. Asking for the
condition number of the resulting coefficient matrix to be relatively small
one estimates an optimum number $M$, $M\leq K-2$, of the coefficients
$h_{0},\ldots,h_{M}$. We refer the reader to \cite[Subsection 3.2 and Section
5]{KT2020} for additional details.

\subsection{Recovery of eigenvalues $\lambda_{n}$ and norming constants
$\alpha_{n}$}

\label{Subsect Norming constants} Suppose we have the coefficients $\omega$,
$\omega_{2}$ and $\{h_{n}\}_{n=0}^{\infty}$ recovered as described in
Subsection \ref{Subsect Recovery hn}. Then we can use the representation
\eqref{Delta NSBF} to calculate the characteristic function $\Delta(\rho^{2})$
for any value of $\rho$. Recalling that the eigenvalues of the problem
\eqref{SL equation}, \eqref{bc1} are real and coincide with zeros of
$\Delta(\lambda)$, the recovery of the eigenvalues reduces to finding zeros of
the entire function $\Delta(\lambda)$.

Having the spectrum $\lambda_{k}=\rho_{k}^{2}$, the corresponding norming
constants can be found using \eqref{alphan from M}, \eqref{Delta0 NSBF} and
\eqref{Delta prime NSBF}. Indeed,
\[
\alpha_{k} = - \frac{1}{\Delta^{0}(\lambda_{k})}\cdot\left.  \frac{d}%
{d\lambda}\Delta(\lambda)\right|  _{\lambda=\lambda_{k}} = -\frac{1}{2\rho_{k}
\Delta^{0}(\rho_{k}^{2})} \left.  \frac{d}{d\rho}\Delta(\rho^{2})\right|
_{\rho=\rho_{k}}.
\]

Suppose we have only a finite number of coefficients $h_{0},\ldots,h_{M}$.
Then approximate eigenvalues are sought as zeros of the function
\begin{equation}
\Delta_{M}(\rho^{2})=-\rho\sin\rho\pi+\omega\cos\rho\pi+2\sum_{n=0}%
^{[M/2]}(-1)^{n}h_{2n}j_{2n}(\rho\pi), \label{Delta approx}%
\end{equation}
and afterwards for each approximate eigenvalue $\lambda_{k}=\rho_{k}^{2}$ the
corresponding norming constant can be obtained from
\begin{equation}
\alpha_{k}\approx\frac{(1+\pi\omega)\sin\rho_{k}\pi-\pi\rho_{k}\cos\rho_{k}%
\pi+2\sum_{n=0}^{[M/2]}(-1)^{n}h_{2n}\left(  \frac{2n}{\rho_{k}}j_{2n}%
(\rho_{k}\pi)-\pi j_{2n+1}(\rho_{k}\pi)\right)  }{2\rho_{k}\cos\rho_{k}%
\pi+2\omega_{2}\sin\rho_{k}\pi+4\sum_{n=0}^{[(M-1)/2]}(-1)^{n}h_{2n+1}%
j_{2n+1}(\rho_{k}\pi)}. \label{alpha_k approx}%
\end{equation}

\subsection{Main system of equations}

Known the eigenvalues $\{\lambda_{n}\}_{n=0}^{\infty}$ and the corresponding
norming constants $\{\alpha_{n}\}_{n=0}^{\infty}$, the potential $q$ and the
parameters $h$ and $H$ can be recovered by solving the Gelfand-Levitan
equation \eqref{Gelfand-Levitan}. In \cite{Kr2019JIIP} (see also
\cite{KrBook2020}) using representation \eqref{K Fourier series} the
Gelfand-Levitan equation was reduced to an infinite system of linear algebraic
equations for the coefficients $\beta_{2n}(x)$. In \cite{KT2020} a more
accurate modification of the method was proposed. We present the final result
from \cite{KT2020} below.

For this subsection and without loss of generality, we suppose that $\rho
_{0}=0$. This always can be achieved by a simple shift of the potential. If
originally $\rho_{0}\neq0$, then we can consider the new potential
$\widetilde{q}(x):=q(x)-\rho_{0}^{2}$. Obviously, the eigenvalues $\left\{
\lambda_{n}\right\}  _{n=0}^{\infty}$ shift by the same amount, while the
numbers $h$, $H$ and $\left\{  \alpha_{n}\right\}  _{n=0}^{\infty}$ do not
change. After recovering $\widetilde{q}(x)$ from the shifted eigenvalues, one
gets the original potential $q(x)$ by adding $\rho_{0}^{2}$ back.

Let us denote
\begin{align}
\widetilde c_{km}(x)  &  = -\frac{\omega x}{8\pi} \left(  \frac{\delta
_{m,k-1}}{(2k-3/2)_{3}} - \frac{2\delta_{m,k}}{(2k-1/2)_{3}} + \frac
{\delta_{m,k+1}}{(2k+1/2)_{3}}\right) \nonumber\\
&  \quad+ (-1)^{k+m}\sum_{n=1}^{\infty}\left[  \frac{j_{2k}(\rho_{n} x)
j_{2m}(\rho_{n} x)}{\alpha_{n}} - \frac{2 j_{2k}(nx) j_{2m}(nx)}{\pi} \right.
\label{c_km tilde}\\
&  \quad+ \left.  \frac{2\omega}{\pi^{2} n}\left(  x j_{2k}(nx) j_{2m+1}(nx) +
x j_{2k+1}(nx)j_{2m}(nx)- \frac{2(k+m)j_{2k}(nx) j_{2m}(nx)}{n}\right)
\right]  ,\nonumber\\
\widetilde C_{0m}(x)  &  = \left(  \frac1{\alpha_{0}} - \frac1\pi
+\frac{2\omega x^{2}}{3\pi^{2}}\right)  \delta_{0m} + \frac{2\omega x^{2}%
}{15\pi^{2}}\delta_{1m} + \widetilde c_{0m}(x),\label{C_0m tilde}\\
\widetilde C_{1m}(x)  &  = \frac{2\omega x^{2}}{15\pi^{2}}\delta_{0m} +
\widetilde c_{1m}(x),\label{C_1m tilde}\\
\widetilde C_{km}(x)  &  = \widetilde c_{km}(x),\qquad k=2,3,\ldots,
\ m\in\mathbb{N}_{0}, \label{C_km tilde}%
\end{align}
and
\begin{equation}
\label{d_k til}%
\begin{split}
\widetilde d_{k}(x)  &  = -\left(  \frac1{\alpha_{0}} - \frac1\pi
+\frac{4\omega x^{2}}{3\pi^{2}}-\frac{\omega x}{\pi}\right)  \delta_{k0} -
\frac{2\omega x^{2}}{15\pi^{2}}\delta_{k1}\\
&  \quad-(-1)^{k}\sum_{n=1}^{\infty}\left[  \frac{\cos\rho_{n}x}{\alpha_{n}%
}j_{2k}(\rho_{n}x)-\frac{2\cos nx}{\pi}j_{2k}(nx)\right. \\
&  \quad+ \left.  \frac{2\omega}{\pi^{2} n} \left(  x \sin nx j_{2k}(nx) +
x\cos nx j_{2k+1}(nx) - \frac{2k}{n} \cos nx j_{2k}(nx)\right)  \right]  ,
\end{split}
\end{equation}
where $\delta_{k,m}$ stands for the Kronecker delta and $(k)_{m}$ is the
Pochhammer symbol.

Then the following result is valid.

\begin{theorem}
[\cite{KT2020}]\label{Thm main system alt} The coefficients $\beta_{2m}(x)$
satisfy the system of linear algebraic equations
\begin{equation}
\frac{\beta_{2k}(x)}{(4k+1)x} + \sum_{m=0}^{\infty}\beta_{2m}(x)\widetilde
C_{km}(x)=\frac{\widetilde d_{k}(x)}2,\qquad k=0,1,\ldots. \label{G-L-alt}%
\end{equation}
For all $x\in\left(  0,\pi\right]  $ and $k=0,1,\ldots$ the series in
\eqref{G-L-alt} converges.
\end{theorem}

It is of crucial importance the fact that for recovering the potential $q$ as
well as the constants $h$ and $H$ it is not necessary to compute many
coefficients $\beta_{2m}(x)$ (that would be equivalent to approximate the
solution of the Gelfand-Levitan equation) but instead the sole $\beta_{0}(x)$
is sufficient for this purpose, see Corollary \ref{Rem q from beta0}.

For the numerical solution of the system \eqref{G-L-alt} it is natural to
truncate the infinite system, i.e., to consider $m,k\leq N$. As was shown in
\cite{KT2020}, the truncation process possesses some very nice properties.
Namely, the truncated system is uniquely solvable for all sufficiently large
$N$, the approximate solutions converge to the exact one, the condition
numbers of the coefficient matrices are uniformly bounded with respect to $N$,
and the solution is stable with respect to small errors in the coefficients.
Moreover, a reduced number of equations in a truncated system is sufficient
for recovering the coefficient $\beta_{0}$ with a high accuracy. However it
should be noted that a lot of (approximate) eigenvalues and norming constants
are necessary to obtain values of the series
\eqref{c_km tilde}--\eqref{d_k til} accurately enough. As was shown in
Subsection \ref{Subsect Norming constants}, it is not a problem, thousands of
approximate eigendata can be computed.

\subsection{Algorithm of solution of Problem \ref{Problem Weyl_countable}}

Given a finite set of points $\left\{  z_{n}\right\}  _{n=0}^{N}=\left\{
\tilde{\rho}_{n}^{2}\right\}  _{n=0}^{N}$ and the corresponding values
$M_{n}=M(z_{n})$, $0\leq n\leq N$ of the Weyl function at these points, the
following direct method for recovering the potential $q$ and the numbers $h$
and $H$ is proposed.

\begin{enumerate}
\item \label{Step1} Solving (overdetermined) system
\eqref{System for xin truncated} find the coefficients $\omega$, $\omega_{2}$
and $h_{0},\ldots,h_{M}$.

\item Compute $h = \omega-\omega_{2}$.

\item \label{Step3} Find approximate eigenvalues $\lambda_{k}$, $k\leq K$ as
zeros of \eqref{Delta approx}. Compute the corresponding norming constants
$\alpha_{k}$, $k\leq K$ using \eqref{alpha_k approx}.

\item \label{Step4} If $\rho_{0}\neq0$, perform the shift of the eigenvalues
\[
\tilde{\rho}_{k}=\sqrt{\rho_{k}^{2}-\rho_{0}^{2}},\qquad k\geq0,
\]
so that $0$ becomes the first eigenvalue of the spectral problem. The
parameters $\omega$ and $\omega_{2}$ have to be shifted as well,
\[
\tilde{\omega}=\omega-\frac{\pi}{2}\rho_{0}^{2},\qquad\tilde{\omega}%
_{2}=\omega_{2}-\frac{\pi}{2}\rho_{0}^{2}.
\]
Let us denote the square roots of the shifted eigenvalues by the same
expression $\rho_{k}$, and similarly for the shifted parameters $\omega$ and
$\omega_{2}$.

\item For a set of points $\left\{  x_{l}\right\}  $ from $(0,\pi]$, compute
the approximate values of the coefficients $\widetilde C_{km}(x)$ and
$\widetilde d_{k}(x)$ for $k,m=0,\ldots,N$ with the aid of the formulas
\eqref{c_km tilde}--\eqref{d_k til} and solve the truncated system
\eqref{G-L-alt} obtaining thus $\beta_{0}(x)$.

\item \label{Step6} Compute $q$ from \eqref{q from beta0}. Take into account
that $\varphi(0,x)$ is an eigenfunction associated with the first eigenvalue
$\lambda_{0}$ and hence does not have zeros on $[0,\pi]$ (see, e. g.,
\cite[Theorem 8.4.5]{Atkinson}). This justifies the division over
$\varphi(0,x)$.

\item Compute $H$ using \eqref{omega}. For this compute the mean of the
potential $\int_{0}^{\pi}q(t)\,dt$, and thus,
\[
H=\omega_{2}-\frac{1}{2}\int_{0}^{\pi}q(t)\,dt.
\]

\item Recall that one has to add the original eigenvalue $\rho_{0}^{2}$ back
to the recovered potential to return to the original potential $q$.
\end{enumerate}

\begin{remark}
The idea of interval flipping proposed in \cite[Subsection 3.6]{KT2020} to
improve the recovery of the potential near $x=\pi$ can be adapted for the
proposed method as well, the formulas (3.39) and (3.40) from \cite{KT2020} are
directly applicable with the data obtained on Steps \ref{Step1}--\ref{Step4}.
We leave the details to the reader.
\end{remark}

\begin{remark}
The proposed method with few modifications can be adapted to the case of
Dirichlet boundary condition at $x=\pi$, corresponding to the value $H=\infty
$. The Weyl function in such case is given by
\[
M_{\infty}(\lambda)=-\frac{S(\rho,\pi)}{\varphi(\rho,\pi)}.
\]
Following the described steps one can similarly obtain an infinite system of
equations for the coefficients $\beta_{n}(\pi)$. The coefficients $\omega$ and
$\omega_{2}$ are no longer necessary. Squares of the zeros of the function
$\varphi(\rho,\pi)$ are the eigenvalues, and the corresponding norming
constants can be obtained similarly. Now one has to solve an inverse problem
by the given spectral density function in the case when one boundary condition
is of the Dirichlet type. It can be done with the use of a Gelfand-Levitan
equation with a different kernel function $F$, see \cite[Chapter 2,
\S 9]{LevitanInverse}, \cite{HT2021} and references therein. We leave the
details for a separate paper.
\end{remark}

Below, in Section \ref{Sect Numerical examples} we illustrate the performance
of this algorithm with several numerical examples.

\section{Numerical examples}

\label{Sect Numerical examples} The proposed algorithm can be implemented
directly, similarly to the algorithm from \cite{KT2020}, only several
observations should be made. On Step \ref{Step1}, to determine the number of
unknowns to look for, we used the following criteria: we try to recover at
least 10 (if less than 30 points $z_{n}$ were given) or at least 20
coefficients $h_{n}$ (for more than 30 points $z_{n}$ given); if the condition
number of the coefficient matrix of the system
\eqref{System for xin truncated} is less than 1000, we increase the number of
recovered coefficients $h_{n}$ until the condition number surpasses 1000. A
significant time saving can be achieved by computing the spherical Bessel
functions $j_{m}(\tilde{\rho}_{k}\pi)$ using a backwards recursion, see
\cite{KT2020}, \cite{Barnett}, \cite{GillmanFiebig} for details.

The number $K$ of the eigenvalues to compute on Step \ref{Step3} can be
estimated from the decay of the coefficients $h_{n}$. Faster decay means
smoother potential, larger value of $K$. Say, $K=10^{4}$ for smooth potentials
(coefficients $h_{n}$ decay very fast) and up to $K=10^{3}$ for non-smooth
potentials (coefficients $h_{n}$ decay slowly). For several examples
originating from the inverse problems considered in Subsection
\ref{Subsect Example Problems}, one can take an arbitrary value of the
parameter $h$ when transforming the Weyl-Titchmarsh $m$-function into the
function $M$ using \eqref{m and M}. In all these examples we took $h=0$. Since
the algorithm recovers two parameters $\omega$ and $\omega_{2}$ which are
equal due to the choice $h=0$, we used the difference $|\omega-\omega_{2}|$ as
some additional indicator of the accuracy of the recovered coefficients.
Smaller value $|\omega-\omega_{2}|$ suggests to take more approximate
eigenvalues (larger value of $K$ used, as was aforementioned).

Similarly to \cite{KT2020}, the number of points $x_{l}$ taken on Step
\ref{Step4} should not be large, about one hundred uniformly spaced points
works best. We used 8 equations in the truncated system
\eqref{System for xin truncated} in all the examples. All the computations
were performed in Matlab 2017 on an Intel i7-7600U equipped laptop computer.
We used spline approximation and differentiated the obtained spline on Step
\ref{Step6}. In all examples where eigenvalues or solution of a Cauchy problem
were necessary, we applied the method from \cite{KNT}.

First we consider several inverse spectral problems which can be reduced to
Problem \ref{Problem Weyl_countable}. In the last subsection we illustrate
that the proposed algorithm can be applied even for complex points $z_{n}$,
but may fail for some particular choices of the points.

\subsection{Two spectra}

\label{Subsect2spectra} Consider an inverse problem of recovering the
potential and boundary parameters by two spectra, see Subsubsection
\ref{Subsubsect Two spectrum}. The same problem was considered in
\cite{KT2020} with the only difference that the shared boundary condition was
at the point $0$ (which is not essential since one can apply the change of
variable $x\leftrightarrow\pi-x$). Since for this problem the given values of
the Weyl function are either 0 or $\infty$, the system \eqref{Equation for hn}
splits into two independent systems. One (given by
\eqref{Equation for hn_infty}) coincides with the system (3.8) in
\cite{KT2020}, but the second system is different from (4.6) from
\cite{KT2020}. Moreover, in \cite{KT2020} we recovered the parameters $\omega$
and $\omega_{2}$ and obtained larger index eigenvalues from the eigenvalue
asymptotics, and used the solutions of the systems only to compute the norming
constants. While in this example we applied the proposed algorithm as is,
without separating the system \eqref{Equation for hn} or previously obtaining
the values of $\omega$ and $\omega_{2}$, and computing all the eigenvalues
from the obtained approximate characteristic function. Of course, one can not
expect obtaining the same accuracy as those of \cite{KT2020}.

We considered three potentials: smooth $q_{1}(x)=\frac{16}{\pi^{2}}x^{2}%
\exp\left(  2-\frac{8x}{\pi}\right)  $ (potential from \cite{Chadan et al
1997}, adapted to the interval $[0,\pi]$), non-smooth continuous
$q_{2}(x)=\bigl|3-|x^{2}-3|\bigr|$ (potential from \cite{KB2008}) and
discontinuous
\[
q_{3}(x)=%
\begin{cases}
0, & x\in\lbrack0,\frac{\pi}{8}]\cup\lbrack\frac{3\pi}{8},\frac{3\pi}{5}),\\
-\frac{12x}{\pi}+\frac{3}{2}, & x\in(\frac{\pi}{8},\frac{\pi}{4}],\\
\frac{12x}{\pi}-\frac{9}{2}, & x\in(\frac{\pi}{4},\frac{3\pi}{8}),\\
4, & x\in\lbrack\frac{3\pi}{5},\frac{4\pi}{5}),\\
2, & x\in\lbrack\frac{4\pi}{5},\pi].
\end{cases}
\]
(potential matching the one from \cite{Chadan et al 1997}). For all potentials
we took $h=1$ and $H=2$.

On Figure \ref{Ex0Fig1} we present the recovered potential $q_{1}$ and its
absolute error. 6, 10 or 16 eigenvalue pairs were used for recovery. As one
can see, the error almost stabilizes. For 16 eigenvalue pairs the parameter
$q$ was obtained with $L_{1}(0,\pi)$ error of $8.6\cdot10^{-7}$, the constants
$h$ and $H$ were obtained with the errors of $4.6\cdot10^{-10}$ and
$2.3\cdot10^{-7}$.

\begin{figure}[p]
\centering
\includegraphics[bb=0 0 216 173
height=2.2in,
width=2.7in
]{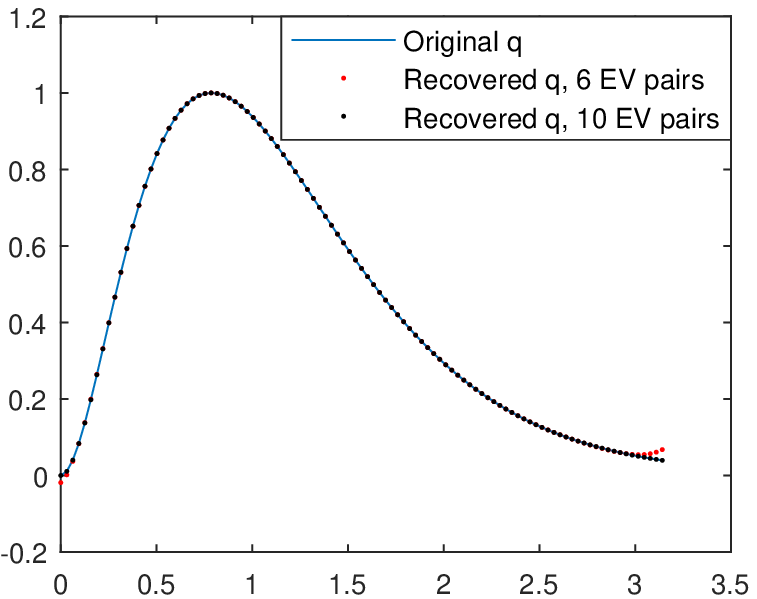} \quad\includegraphics[bb=0 0 216 173
height=2.2in,
width=2.7in
]{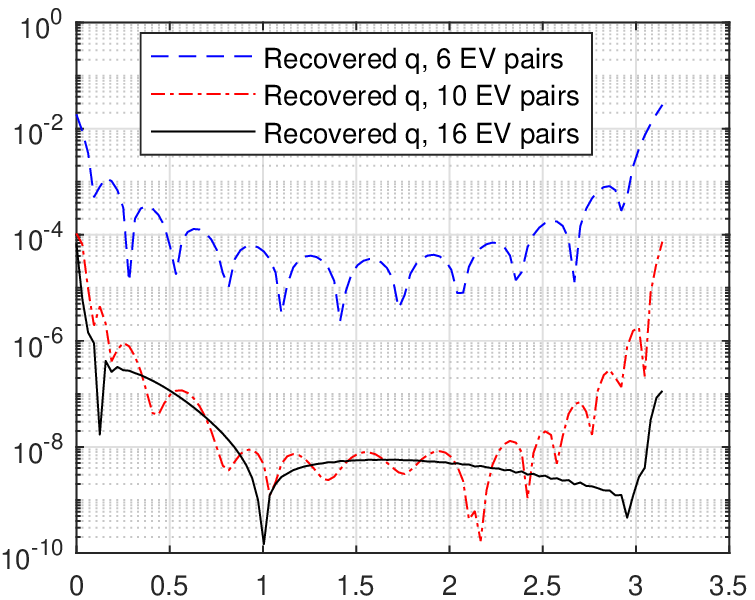}\caption{On the left: exact (blue line), recovered from 6
eigenvalue pairs (red dots) and recovered from 10 eigenvalue pairs (black
dots) potential $q_{1}$ from Subsection \ref{Subsect2spectra}. On the right:
absolute error of the recovered potential, 6, 10 or 16 eigenvalue pairs were
used. }%
\label{Ex0Fig1}%
\end{figure}

On Figure \ref{Ex0Fig2} we present the potential $q_{2}$ recovered from 40 and
201 eigenvalue pairs. For 40 eigenvalue pairs the potential $q$ was obtained
with $L_{1}(0,\pi)$ error of $4.4\cdot10^{-2}$, the constants $h$ and $H$ were
obtained with the errors of $0.017$ and $0.024$. For 201 eigenvalue pairs,
with $1.2\cdot10^{-2}$, $0.011$ and $0.014$, respectively.

\begin{figure}[p]
\centering
\includegraphics[bb=0 0 216 173
height=2.2in,
width=2.7in
]{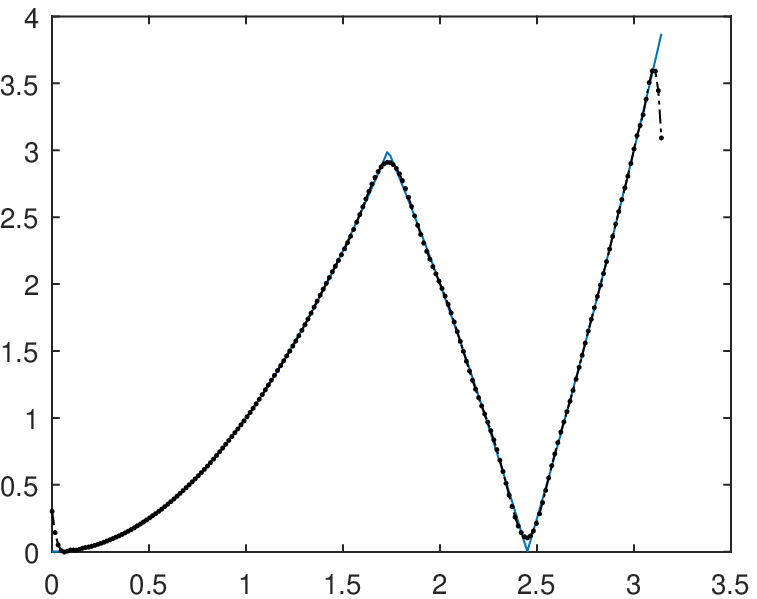} \quad\includegraphics[bb=0 0 216 173
height=2.2in,
width=2.7in
]{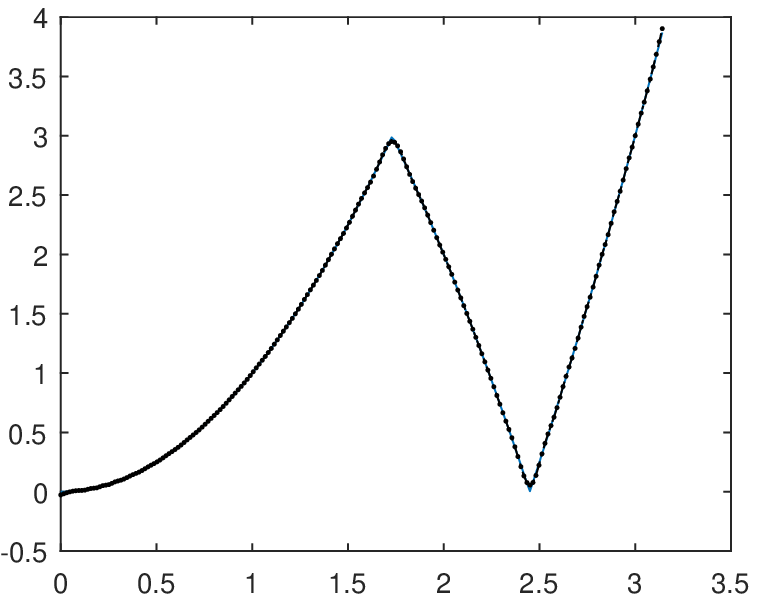}\caption{On the left: exact (blue line) and recovered from 40
eigenvalue pairs (black dots); on the right: exact (blue line) and recovered
from 201 eigenvalue pairs (black dots) potential $q_{2}$ from Subsection
\ref{Subsect2spectra}. }%
\label{Ex0Fig2}%
\end{figure}

On Figure \ref{Ex0Fig3} we present the potential $q_{3}$ recovered from 30 and
201 eigenvalue pairs. While for 30 eigenvalue pairs the algorithm failed near
the interval endpoints, inside the interval $(0.5, \pi-0.5)$ the result is
acceptable. For 201 eigenvalue pairs the algorithm was able to recover both
potential and boundary parameters $h$ and $H$, $L_{1}$ error of the potential
was $0.24$, absolute errors of the boundary parameters were $0.0076$ and
$0.029$, respectively.

\begin{figure}[p]
\centering
\includegraphics[bb=0 0 216 173
height=2.2in,
width=2.7in
]{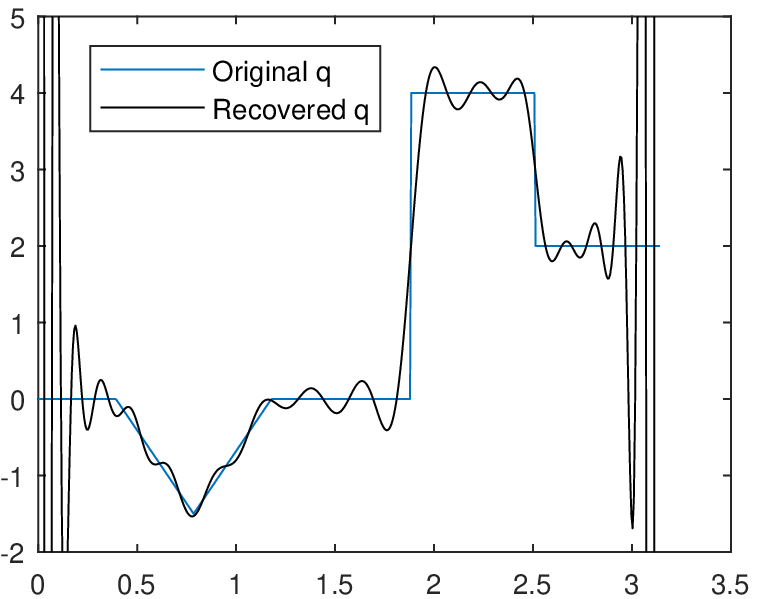} \quad\includegraphics[bb=0 0 216 173
height=2.2in,
width=2.7in
]{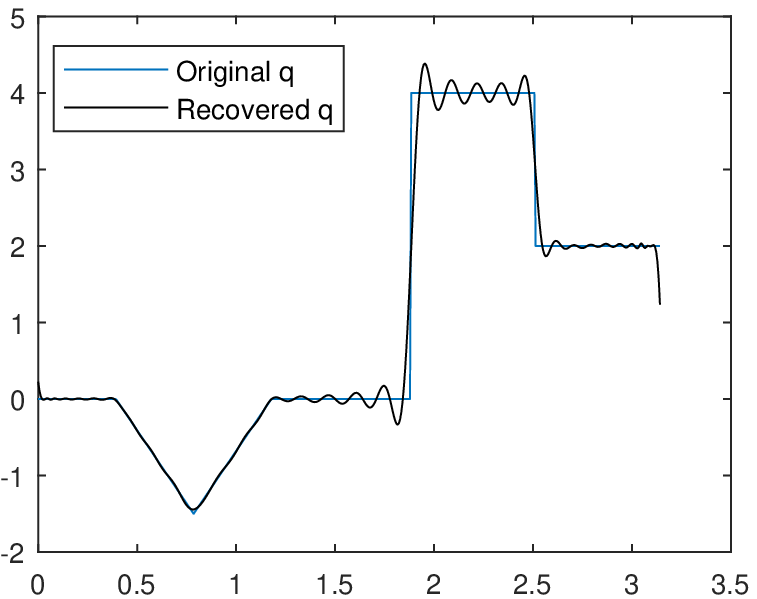}\caption{On the left: exact (blue line) and recovered from 30
eigenvalue pairs (black dots); on the right: exact (blue line) and recovered
from 201 eigenvalue pairs (black dots) potential $q_{3}$ from Subsection
\ref{Subsect2spectra}. }%
\label{Ex0Fig3}%
\end{figure}

\subsection{Inverse problems with known parts of several spectra}

\label{Subsect3spectra} As it is stated in \cite{GRS1997}, ``Two-thirds of the
spectra of three spectral problems uniquely determine $q$.'' In this
subsection we numerically illustrate the solution of such inverse problems.

To illustrate the performance of the algorithm we considered two potentials,
the same smooth potential $q_{1}$ and the following potential, possessing a
discontinuous second derivative $q_{4}(x)=\int_{0}^{x}q_{2}(s)\,ds$, where
$q_{1}$ and $q_{2}$ were introduced in Subsection \ref{Subsect2spectra}. We
took $H=2$ and $h\in\{1,2,3\}$, and eigenvalues indexed $0,1,3,4,6,7,\ldots$
from the first spectrum, $0,2,3,5,6,8,9,\ldots$ from the second spectrum and
$1,2,4,5,7,8,\ldots$ from the third spectrum.

\begin{figure}[p]
\centering
\includegraphics[bb=0 0 216 173
height=2.2in,
width=2.7in
]{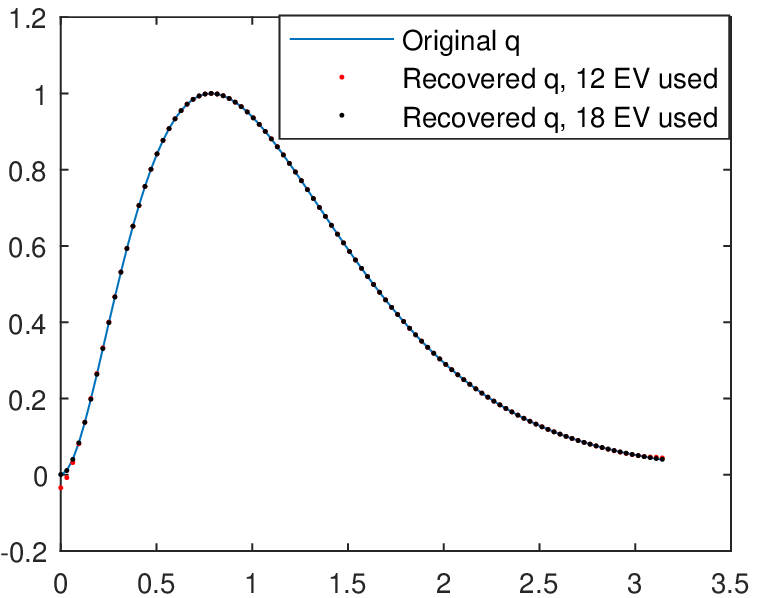} \quad\includegraphics[bb=0 0 216 173
height=2.2in,
width=2.7in
]{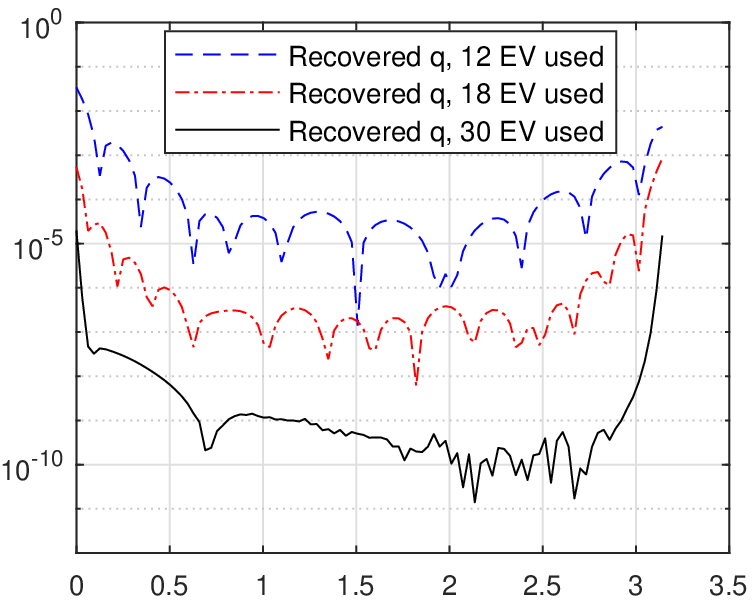}\caption{On the left: exact (blue line), recovered from 12
eigenvalues (red dots) and recovered from 18 eigenvalues (black dots)
potential $q_{1}$ from Subsection \ref{Subsect3spectra}. On the right:
absolute error of the recovered potential, 12, 18 or 30 eigenvalues were used.
}%
\label{Ex1Fig1}%
\end{figure}

On Figure \ref{Ex1Fig1} we present the recovered potential $q_{1}$ and its
absolute error. The absolute error rapidly decreases with more eigenvalues
used and stabilizes at 24 eigenvalues (8 from each spectra). $L_{1}$ error of
the recovered $q_{1}$ was $3.9\cdot10^{-7}$, the parameter $H$ had an absolute
error of $3.1\cdot10^{-7}$.

\begin{figure}[p]
\centering
\includegraphics[bb=0 0 216 173
height=2.2in,
width=2.7in
]{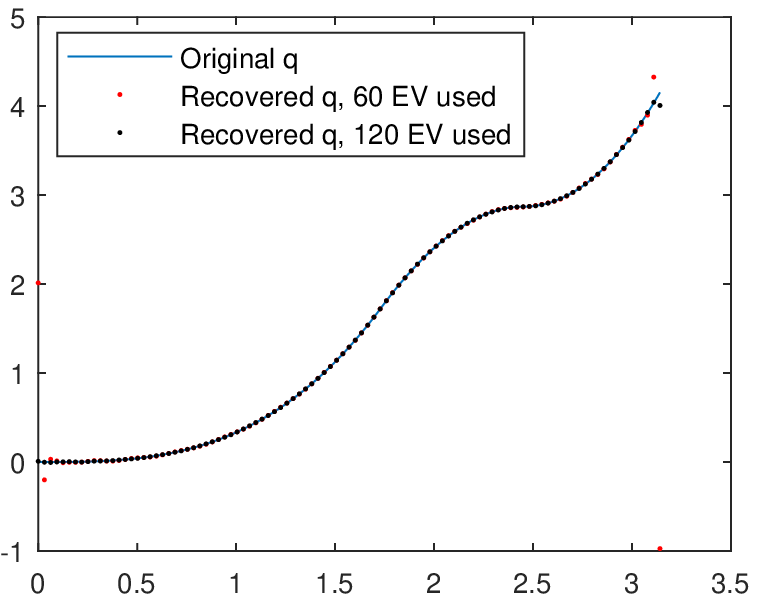} \quad\includegraphics[bb=0 0 216 173
height=2.2in,
width=2.7in
]{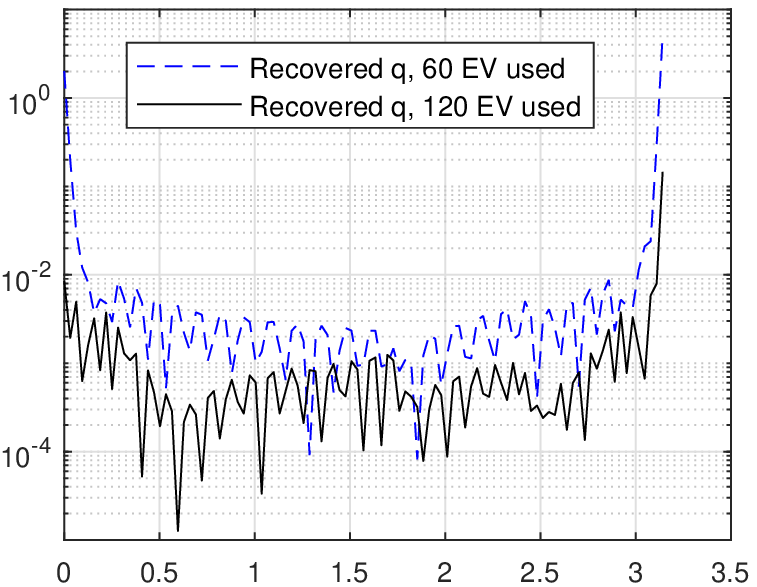}\caption{On the left: exact (blue line), recovered from 60
eigenvalues (red dots) and recovered from 120 eigenvalues (black dots)
potential $q_{4}$ from Subsection \ref{Subsect3spectra}. On the right:
absolute error of the recovered potential, 60 or 120 eigenvalues were used. }%
\label{Ex1Fig2}%
\end{figure}

\begin{figure}[p]
\centering
\includegraphics[bb=0 0 216 173
height=2.2in,
width=2.7in
]{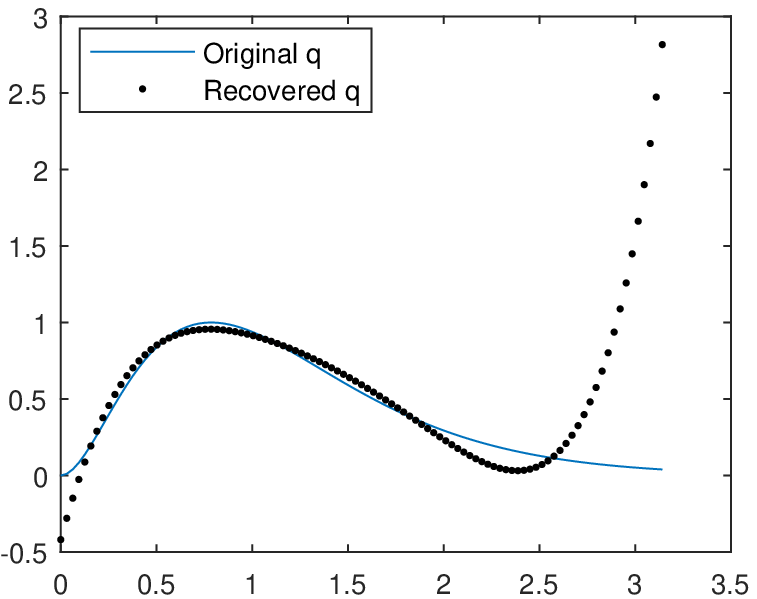} \quad\includegraphics[bb=0 0 216 173
height=2.2in,
width=2.7in
]{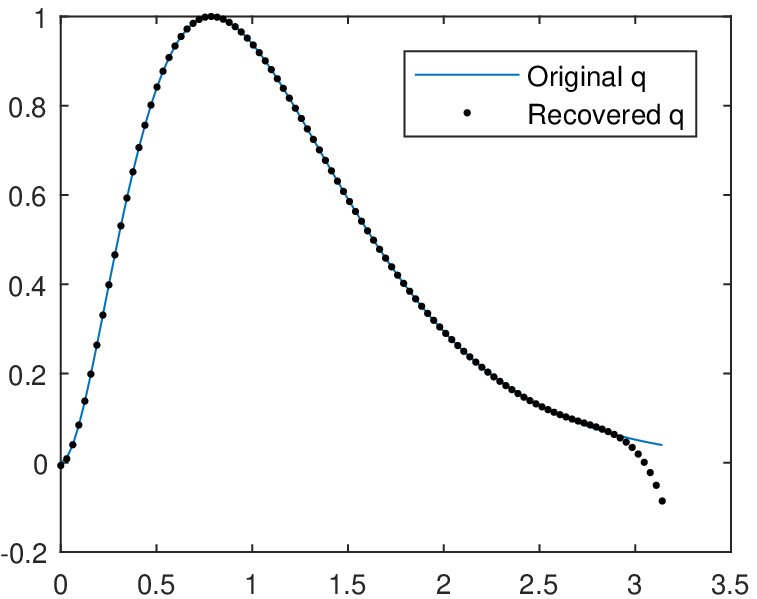}\caption{Exact potential $q_{1}$ from Subsection
\ref{Subsect3spectra} (blue line) and recovered one (black dots). On the left:
recovered from 10 eigenvalues $\lambda_{0}(h)$, $h\in\{0,1,2,3,4,5, 10, 20,
50, 100\}$, on the right: recovered from 30 eigenvalues $\lambda_{0}(h),
\lambda_{1}(h), \lambda_{2}(h)$, $h\in\{0,1,2,3,4,5, 10, 20, 50, 100\}$. }%
\label{Ex1Fig3}%
\end{figure}

On Figure \ref{Ex1Fig2} we present the recovered potential $q_{4}$ and its
absolute error. In this case the potential is of finite smoothness, so the
algorithm converges slowly. With 120 eigenvalues (40 from each spectrum) we
obtained the following errors: $L_{1}$ error of the potential was
$3.7\cdot10^{-3}$ and the error of the parameter $H$ was $1.5\cdot10^{-3}$.

In \cite[Section 3.14]{Chadan et al 1997} the following inverse problem was
considered. Suppose that the first eigenvalue is known for spectral problems
having the same potential but different boundary conditions at $x=0$. From
countably many values the potential can be recovered uniquely. However, if
only finitely many values are given, the situation changes. The corresponding
numerical problem is extremely ill-posed, so the method from \cite{Rundell
Sacks} was not able to solve it, and the modification from \cite{Chadan et al
1997} was able to recover only the general shape of the potential.

We tried the proposed algorithm on this problem. For a numerical example we
took the potential $q_{1}$ from Subsection \ref{Subsect2spectra}, $H=0$ and
$h\in\{0,1,2,3,4,5, 10, 20, 50, 100\}$. On Figure \ref{Ex1Fig3}, left plot, we
present the recovered potential. The algorithm failed, it is a case of very
small interval for $z_{k}$. And the result does not improve if we take more
values of the parameter $h$. However the situation greatly improves if instead
of 1 eigenvalue for each spectral problem we take several. On Figure
\ref{Ex1Fig3}, right plot, we present the recovered potential for the inverse
problem when the first 3 eigenvalues were taken from each spectrum.

\subsection{Partially known potential}

\label{SubsectPartial} In this subsection we illustrate the performance of the
proposed method applied to solution of inverse problems with a partially known
potential. For comparison we have taken two potentials from \cite{KB2008}, a
smooth potential $q_{5}(x)=\frac{e^{x}-x^{2}}{12}$ and a non-smooth continuous
potential $q_{2}(x)=|3-|x^{2}-3||$. The potentials are known on $[0,\pi/2]$,
hence one spectrum is necessary to recover the potential. While in
\cite{KB2008} the Dirichlet boundary conditions are considered, we decided to
take mixed boundary conditions with $h=1$ and $H=2$ for both potentials.

To solve the inverse problem we followed the idea presented in Subsubsection
\ref{SubsubsectPartial}: solve the Cauchy problems with initial data
$v(\rho_{n},0)=1$, $v^{\prime}(\rho_{n},0)=h$ on the segment $[0,\pi/2]$,
compute the values of the Weyl-Titchmarsh $m$-function
\[
m(\pi/2,\rho_{n}^{2})=\frac{v^{\prime}(\rho_{n},\pi/2)}{v(\rho_{n},\pi/2)},
\]
transform them to the Weyl function $M$ using \eqref{m and M}. After a simple
rescaling we get Problem \ref{Problem Weyl_countable}.

For the numerical example 40 eigenvalues with 1 missing eigenvalue were
considered in \cite{KB2008}. For the potential $q_{5}$ we similarly considered
40 eigenvalues, but with 2 eigenvalues missing: $\lambda_{10}$ and
$\lambda_{35}$. For the potential $q_{2}$ one eigenvalue $\lambda_{35}$ was
missing. We would like to emphasize that we used the proposed method
\textquotedblleft as is\textquotedblright, without recovering missing
eigenvalues (which is impossible, since the algorithm has no information if
the points $z_{n}$ correspond to any eigenvalues).

On Figures \ref{Ex2Fig1} and \ref{Ex2Fig2} we present the recovered potentials
and their absolute errors. One can appreciate an especially excellent accuracy
in the case of a smooth potential.

\begin{figure}[p]
\centering
\includegraphics[bb=0 0 216 173
height=2.2in,
width=2.7in
]{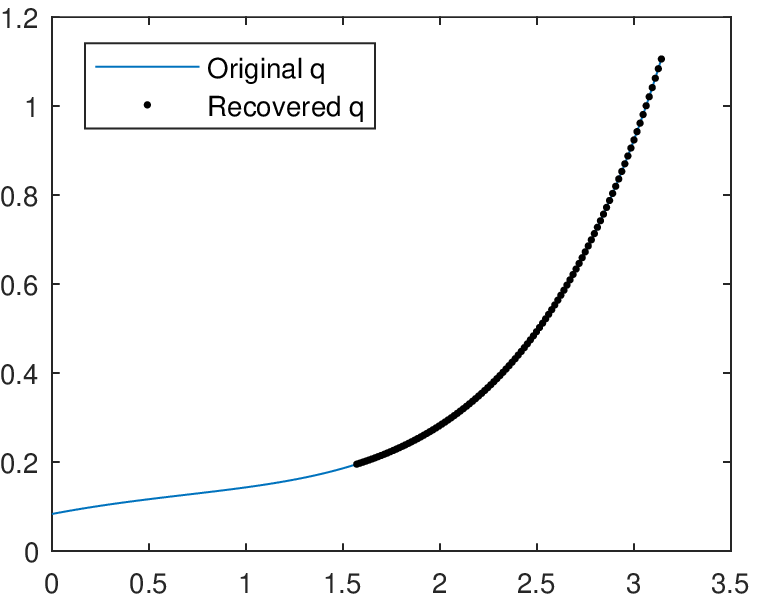} \quad\includegraphics[bb=0 0 216 173
height=2.2in,
width=2.7in
]{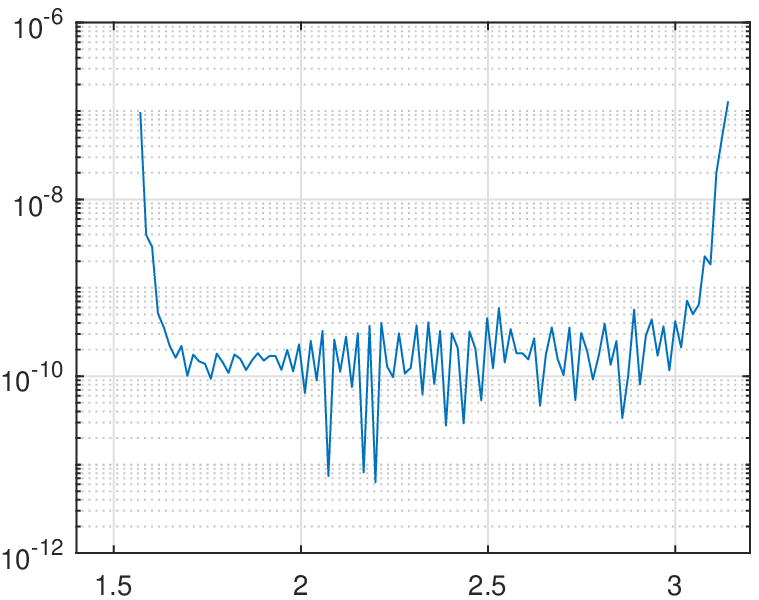}\caption{On the left: exact (blue line) and recovered (black
dots) potential $q_{5}$ from Subsection \ref{SubsectPartial}. On the right:
absolute error of the recovered potential. Eigenvalues $\{\lambda_{n}%
\}_{n=0}^{39}\setminus\{\lambda_{10},\lambda_{35}\}$ were used. The error of
the recovered parameter $H$ is $5.0\cdot10^{-10}$, and the $L_{1}$ error of
the recovered potential is $3.0\cdot10^{-9}$. }%
\label{Ex2Fig1}%
\end{figure}

\begin{figure}[p]
\centering
\includegraphics[bb=0 0 216 173
height=2.2in,
width=2.7in
]{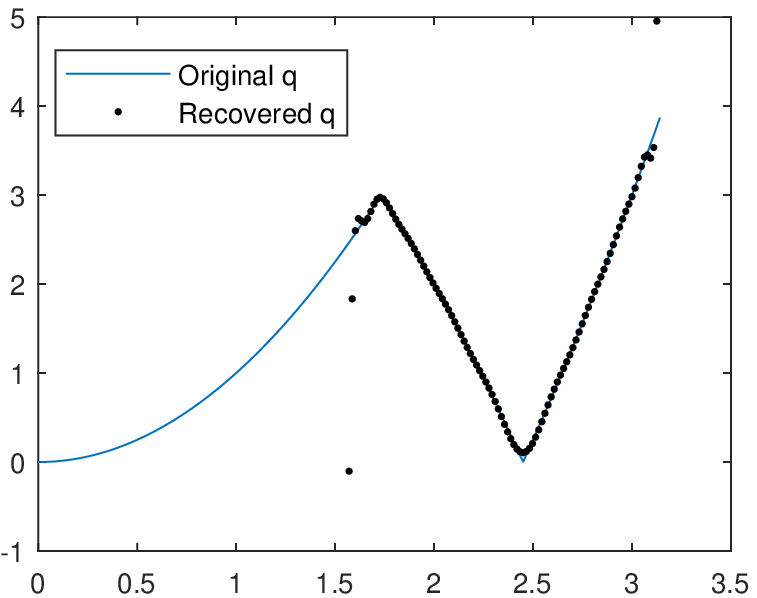} \quad\includegraphics[bb=0 0 216 173
height=2.2in,
width=2.7in
]{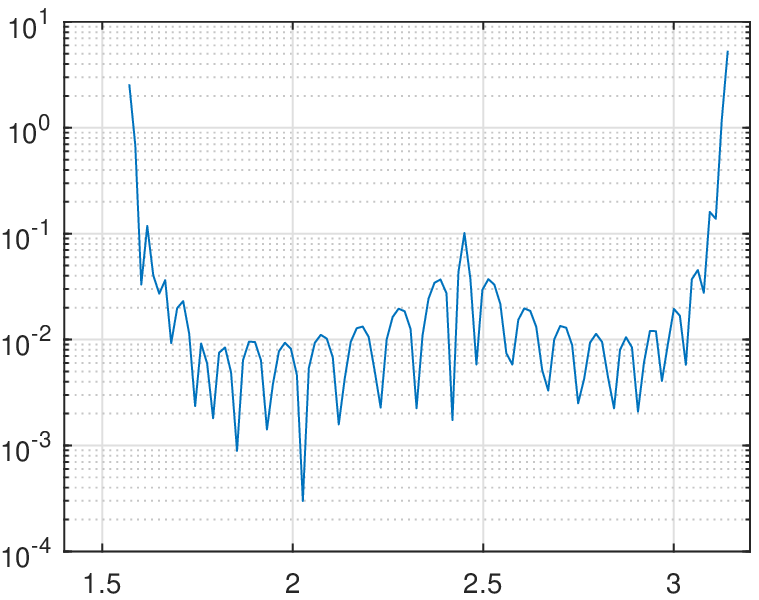}\caption{On the left: exact (blue line) and recovered (black
dots) potential $q_{2}$ from Subsection \ref{SubsectPartial}. On the right:
absolute error of the recovered potential. Eigenvalues $\{\lambda_{n}%
\}_{n=0}^{39}\setminus\{\lambda_{35}\}$ were used. The error of the recovered
parameter $H$ is $2.5\cdot10^{-2}$, and the $L_{1}$ error of the recovered
potential is $0.11$.}%
\label{Ex2Fig2}%
\end{figure}

\subsection{Recovery of the potential from values of the Weyl function:
complex points $z_{n}$ and some problematic choices}

\label{Subsect Recovery from M} The proposed algorithm works even when the
numbers $z_{n}$ are complex (however for numerical stability the imaginary
parts should be relatively small comparing to the real parts). For
illustration we considered the potential $q_{1}$ from Subsection
\ref{Subsect2spectra}, $h=1$, $H=2$ and the following sets of points
\begin{equation}
z_{n}^{(k)}=\left(  \frac{1}{5}+\frac{n}{2}+k\cdot i\right)  ^{2},\qquad0\leq
n\leq40,\ k\in\{0,1,2\}. \label{z_n complex}%
\end{equation}
The obtained results are presented on Figure \ref{Ex3Fig1}, left plot. The
algorithm was able to recover both the potential and the boundary conditions,
however the accuracy decreases when the points are taken farther from the real line.

\begin{figure}[p]
\centering
\includegraphics[bb=0 0 216 173
height=2.2in,
width=2.7in
]{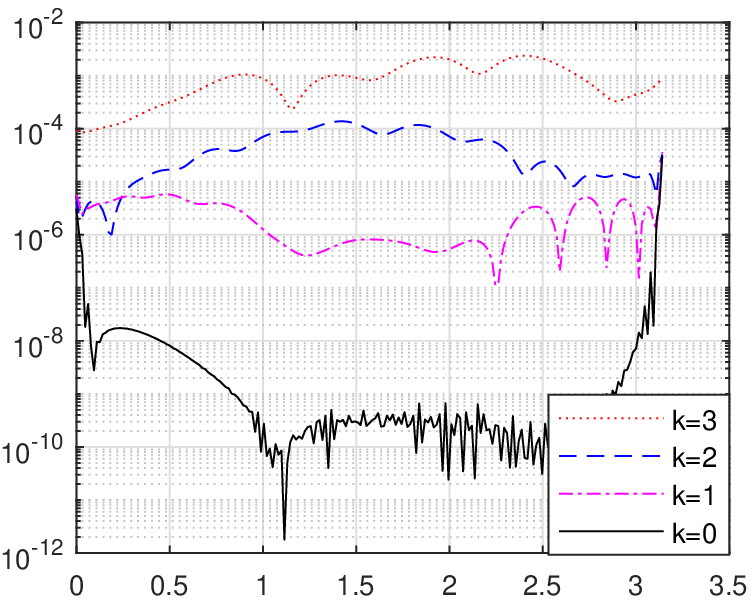} \quad\includegraphics[bb=0 0 216 173
height=2.2in,
width=2.7in
]{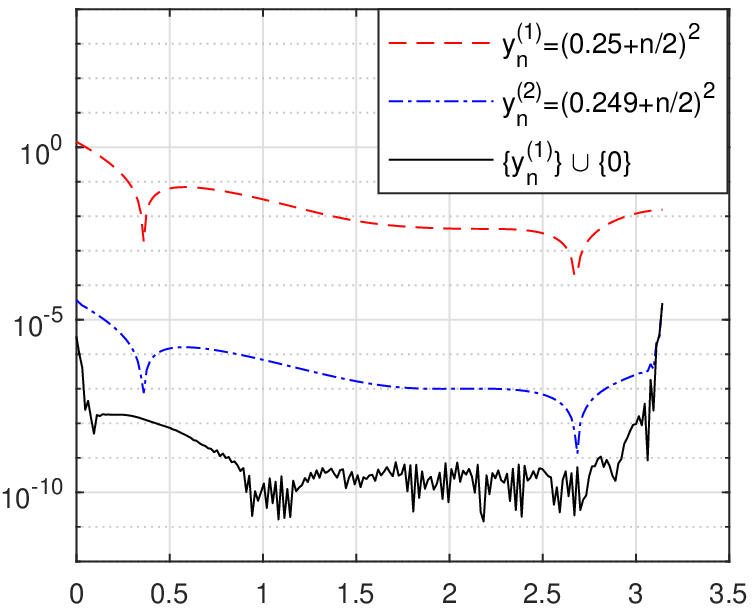}\caption{Absolute errors of the recovered potential $q_{1}$
from Subsection \ref{Subsect Recovery from M}. On the left: using the points
\eqref{z_n complex} with different values of $k$. On the right: using the
points shown on the legend, $0\le n\le40$.}%
\label{Ex3Fig1}%
\end{figure}

It should be mentioned that the proposed algorithm may fail for some
particular choices of the points $z_{n}$ even for nice potentials. We
considered the same problem and the following set of points
\[
y_{n}^{(1)}=\left(  \frac{1}{4}+\frac{n}{2}\right)  ^{2},\qquad0\leq n\leq40.
\]
The truncated system \eqref{System for xin truncated} was not ill conditioned,
but its solution was nowhere close to the exact one and the situation does not
improve when taking more or less unknowns. For example, for 9 unknowns
$\xi_{0},\ldots,\xi_{8}$ the parameter $h=\omega-\omega_{2}$ resulted to be
$0.7999$ instead of $1$. The best value was achieved for 10 unknowns and that
was $0.8677$. On Figure \ref{Ex3Fig1}, right plot we present the absolute
error of the recovered potential in comparison with those recovered using a
very close set of points
\[
y_{n}^{(2)}=\left(  0.249+\frac{n}{2}\right)  ^{2},\qquad0\leq n\leq40.
\]
For the latter choice of points, the recovered parameter $h$ resulted to be
$0.999997$ and the potential was recovered to 5 extra decimal digits. It
should be noted that the sets $\{y_{n}^{(1)}\}_{n=0}^{\infty}$ and
$\{y_{n}^{(2)}\}_{n=0}^{\infty}$ do not satisfy neither the condition
\eqref{Condition on lambda n} nor the condition from \cite{Horvath2005}, so any behavior can be expected for finite
subsets. And the set $\{y_{n}^{(1)}\}_{n=0}^{\infty}$ is, in some sense, the
worst possible since the points $y_{n}^{(1)}$ are the most distant from any
eigenvalue asymptotics (numbers close either to integers or integers plus one
half). The situation changes completely if we consider the set $\{0\}\cup
\{y_{n}^{(1)}\}_{n=0}^{\infty}$. This set satisfies the condition
\eqref{Condition on lambda n}, and the proposed method works for it without
any problem, see Figure \ref{Ex3Fig1}, right plot. The recovered parameter $h$
was $0.9999999991$.

\section{Conclusions}

A direct method for recovering the Sturm-Liouville problem from the
Weyl-Titchmarsh function given on a countable set of points is developed,
which leads to an efficient numerical algorithm for solving a variety of
inverse Sturm-Liouville problems. The main role in the proposed approach is
played by the coefficients of the Fourier-Legendre series expansion of the
transmutation operator integral kernel. We show how the given spectral data
lead to an infinite system of linear algebraic equations for the coefficients,
and the crucial observation is that for recovering the potential it is
sufficient to find the first coefficient only. In practical terms this
observation translates into the fact that very few equations in the truncated
system are enough for solving the inverse Sturm-Liouville problem.

The method is simple, direct and accurate. Its limitations are all related to
a possible insufficiency of the input data. A too reduced number of eigendata
or a special choice of points at which the Weyl function is given may lead to
inaccurate results, as we illustrate by numerical examples.

\section*{Acknowledgements}

Research was supported by CONACYT, Mexico via the project 284470. Research of
Vladislav Kravchenko was supported by the Regional mathematical center of the
Southern Federal University, Russia.

\end{document}